\documentclass[12pt, a4paper]{article}
\usepackage{mdwlist}
\usepackage{amsmath}
\usepackage{amsthm}
\usepackage{amssymb}
\usepackage{amsfonts}
\usepackage{color}
\usepackage[all]{xy}
\usepackage{enumitem}
\newcommand{\nd}{\triangleleft}
\newcommand{\Q}{\ensuremath{\mathbb{Q}}}
\newcommand{\Z}{\ensuremath{\mathbb{Z}}}
\newcommand{\N}{\ensuremath{\mathbb{N}}}
\newcommand{\A}{\ensuremath{\mathbb{A}}}
\newcommand{\R}{\ensuremath{\mathbb{R}}}
\newcommand{\F}{\ensuremath{\mathbb{F}}}
\newcommand{\E}{\ensuremath{\mathbb{E}}}
\newcommand{\oo}{\ensuremath{\mathsf{O}}}
\newcommand{\co}{\ensuremath{\mathcal{O}}}
\newcommand{\cu}{\ensuremath{\mathcal{U}}}
\newcommand{\cs}{\ensuremath{\mathcal{S}}}
\newcommand{\ct}{\ensuremath{\mathcal{T}}}
\newcommand{\cg}{\ensuremath{\mathcal{G}}}
\newcommand{\gl}{\mathfrak{g}}
\newcommand{\ml}{\mathfrak{m}}
\newcommand{\tl}{\mathfrak{t}}
\newcommand{\ul}{\mathfrak{u}}
\newcommand{\cst}{\Upsilon}
\newcommand{\fhs}{\psi}
\newtheorem{thm}{}
\newtheorem{ste}[thm]{Theorem}
\newtheorem{mdef}[thm]{Definition}
\newtheorem{lem}[thm]{Lemma}
\newtheorem{gev}[thm]{Corollary}
\newtheorem{ver}[thm]{Conjecture}
\newtheorem{pro}[thm]{Proposition}

\begin{document}
\begin{center}
 \textbf{\Large Local Summability of Characters on $p$-adic Reductive Groups\\}
 \vspace{0.3cm}
{\large Julius Witte}\\
\vspace{0.1cm}
Radboud Universiteit Nijmegen\\
Heyendaalseweg 135, 6525AJ Nijmegen, the Netherlands\\
email: J.Witte@math.ru.nl
\end{center}

\tableofcontents
\vspace{12pt}
\noindent\textbf{Acknowledgements}\\
The author would like to thank M. Solleveld for many useful discussions. The author is financially supported by NWO grant nr 613.009.033 and EW wiskundecluster GQT positie. 
\section*{Abstract}
In this paper we study the complex representations of reductive groups over local non-Archimedean fields. We use the building of the reductive group to give upper-bounds for the absolute value of the character of an admissible representation and for the Weyl integration formula for certain regular elements. The upper-bound for the character of a representation is based on the alternative description, depending on the building, of the character as given by R. Meyer and M. Solleveld \cite{MS12}. Once the character and the Weyl integration formula are related to the building, the upper-bounds will follow from a similar argument. Both upper-bounds generalize the upper-bounds given by Harish-Chandra \cite{HC70} to groups defined over fields of positive characteristic. At last following Harish-Chandra's method we combine both upper-bounds to show that for a maximal torus $T$ containing a maximal split torus the character is locally summable on $\{ gtg^{-1} : g\in G,t\in T\}$.
\section{Introduction}
Let $\F$ be a non-Archimedean local field with characteristic $p$ and residue field of order $q$. Let $G$ be a reductive group over $\F$. Let $\pi$ be a complex admissible representation of $G$. Let $\theta$ be the character of the representation $\pi$.
\begin{ver}
 $\theta$ is locally integrable on $G$.
\end{ver}
In the case that $\F$ has characteristic $0$ this conjecture has been proven by Harish-Chandra, see \cite{HC99}. He transports the problem to the Lie algebra  with the exponential map. On the Lie algebra he shows that the $\theta$ can locally be written as a linear combination of Fourier transforms of nilpotent orbital integrals. Since the Fourier transforms of nilpotent orbital integrals are locally summable that completes the proof. Up to and including the moment of writing this paper the author is not aware of a proof of the conjecture for general $\F$ and $G$. There has been made progress on proving the conjecture in two directions. It has been shown that the conjecture is true for particular groups, e.g. $SL_n$ \cite{Le96} and $GL_n$ \cite{Le05,Ro85}. Also for every group $G$ defined over $\Z$ there is an $N$ such that if $p>N$, then the conjecture holds \cite{CGH14a}. In both \cite{Ro85} and \cite{CGH14a} one more or less generalizes the proof given by Harish-Chandra to fields of positive characteristic. One follows the proof of Harish-Chandra to show that the trace is a linear combination of Fourier transforms of nilpotent orbital integrals, to prove the conjecture when the characteristic is large enough. For each step in the proof one tries to generalize this step to positive characteristic and/or keep track of the assumptions made, see for example \cite{DB02a}. That the nilpotent distributions are locally summable in large positive characteristic is shown by motivic integration in \cite{CGH14a}. Here one shows that $\theta$ is locally summable in characteristic $0$ if and only if it is locally summable for all large $p$.\\

Let $g\in G$ be semi-simple. Define $D(g)$ to be the Harish-Chandra $D$-function: Let $T$ be a maximal torus containing $g$. Define \\ $D(g) := \prod_{\alpha\in R(G,T)} (\alpha(g)-1)$, where $R(G,T)$ is the root system of $T$ and $G$.\\
Let $\lambda(g)$ be such that $q^{\lambda(g)} = |D(g)|$.\\

If $\pi$ is a cuspidal representation Harish-Chandra proves (in characteristic $0$) that $\theta$ is locally summable in an other way, see \cite{HC70}. His proof consists of four steps:
\begin{enumerate}
 \item For every $g\in G$ there exist a compact neighborhood $\omega$ of $g$, a $C\in\R_{>0}$ and $n\in \N$ such that for all $\gamma\in \omega$
$$|\theta(\gamma)|\leq C|\lambda(\gamma)|^n|D(\gamma)|^{-\frac12}.$$
 \item For all $\epsilon\geq 0$
 \begin{multline*}
 \int_{^GT}|D(x)|^{-\frac12-\epsilon} f(x)dx\\
 =|N_G(T)/T|^{-1}\int_T |D(t)|\int_{G/T}|D(t)|^{-\frac12-\epsilon}f(gtg^{-1})dgdt.
 \end{multline*}
 \item For every $f\in C^\infty_c(G)$ there exists a $C\in\R_{> 0}$ such that for all regular $t\in T$
 \[\int_{G/T}f(gtg^{-1})dg\leq C|D(t)|^{-\frac12}.\]
 \item For small $\epsilon\geq 0$ the function $|\lambda(t)|^n|D(t)|^{-\epsilon}$ is locally summable on $T$.
\end{enumerate}
The local summability of the character on $G$ follows from these four statements, because, when the characteristic is $0$, there are only finitely many conjugacy classes of maximal $\F$-tori:
\begin{align}
\begin{split}\label{berinto}
 \int_{^GT}|f\theta|dg &\leq \int_{^GT}C_\theta|\lambda(g)|^n|D(g)|^{-\frac{1}{2}} |f(x)|dx\\ 
 &= |W|^{-1}\int_T c_\theta |D(t)|^\frac12\int_{^GT} |\lambda(gtg^{-1}|^n|f(gtg^{-1})|dgdt\\
 &\leq \frac{C_\theta C_f}{|W|}\int_T|
 \lambda(t)|^ndt \leq C_\epsilon\int_T|D(t)|^{-\epsilon}dt
\end{split}
\end{align}

In this paper we give similar estimates as in statements 1 and 3 in the case that $\gamma\in Z_G(S)$, and we prove statement 2 and 4. The advantage of our method is that it also works in positive characteristic. To be more precise we will proof the following:
\begin{ste}
For all maximal tori $T$ of $G$:
\begin{enumerate}
  \item For all $\epsilon\geq 0$
 \begin{multline*}
 \int_{^GT}|D(x)|^{-\frac12-\epsilon} f(x)dx\\
 =|N_G(T)/T|^{-1}\int_T |D(t)|\int_{G/T}|D(t)|^{-\frac12-\epsilon}f(gtg^{-1})dgdt.
 \end{multline*}
  \item For small $\epsilon> 0$ the function $sd(t)^n|D(t)|^{-\epsilon}$ is locally summable on $T$.
\end{enumerate}
Let $S$ be a maximal $\F$-split torus and $\Phi$ the roots of $S$ and $G$.
\begin{enumerate}
\setcounter{enumi}{2}
 \item For every $g\in G$ there exist a compact neighborhood $\omega$ of $g$, a $C\in\R_{>0}$ and $n\in \N$ such that for all $\gamma\in \omega\cap {^GZ_G(S)}$
\[|\theta(\gamma)|\leq C(ht(\Phi)sd(\gamma))^n|D(\gamma)|^{-\frac12}.\]
\end{enumerate}
If moreover $T \subset Z_G(S)$, then
\begin{enumerate}
\setcounter{enumi}{3}
 \item For every $f\in C^\infty_c(G)$ there exists a $C\in\R_{> 0}$ such that for all regular $t\in T$
 \[\int_{G/T}f(gtg^{-1})dg\leq C|D(t)|^{-\frac12}.\]
\end{enumerate}

\end{ste}
Here $sd(\gamma)$ denotes the singular depth of $\gamma$, which measures how singular $\gamma$ is.\\

The first statement follows directly from the Weyl integration formula. This formula is well-known if char $\F=0$, but the author could not find a good reference for the general case. Therefore a proof of the Weyl integration formula is added to this paper in the appendix.\\

As the calculation (\ref{berinto}) shows, we get the following Theorem as consequence.

\begin{ste}
 Let $(\rho,V)$ be a $G$-representation of finite length with character $\theta$ and $f\in C_c^\infty (G)$, then for every torus $T$ containing a maximal $\F$-split torus:
 \[\int_{^GT} f(g)\theta(g)dg<\infty.\]
\end{ste}

Assume that $\gamma\in Z_G(S)$ is compact. We use an alternative description of the character, which uses the reduced Bruhat-Tits building of the reductive group $G$, given by Meyer and Solleveld in $\cite{MS10}$ and $\cite{MS12}$, for the local upper-bound of the character. The non-compact case is deduced from the compact case via Casselman's method and the displacement function. For the upper-bound of the Weyl integral $\int_{G/T}f(g\gamma g^{-1})dg$ we use the extended and the reduced buildings. Both estimates are related to the fixed points of $\gamma$ in a reduced building.

After giving a definition and notation for the extended and reduced building, we study the distribution of $\gamma$-fixed points in the reduced building. Then we give an upper-bound for the trace of a representation with finite level. After that section we give an upper-bound for $\int_{G/T}f(gtg^{-1})dg$. In the last section we combine both upper-bounds to a proof of the local summability of $\theta$ on $\{gtg^{-1} : g\in G,t\in T\}$.

Most of the Lemma's and Theorems about the fixed points in the building and the relation between the Weyl integral and the fixed points are inspired by examples such as $SL_2(\F)$ and $GL_3(\F)$.
\section{Notations}
Let $\F$ be a local non-archimedean field with valuation $v : \F^\times \rightarrow \R$, ring of integers $\co$ and uniformizer $\pi$. Define $q$ to be the order of the residue field of $\F$. Let $p$ be the characteristic of $\F$. Let $k$ be an algebraic closure of $\F$.\\
$\cg,\cs,\ct,\cu$ are linear algebraic groups over $\F$ and $G,S,T,U$ are the $\F$-points of these groups respectively. The Lie algebra of a group $G$ is denoted by $\gl$. $\cg$ is a connected reductive group and $\ct$ is a maximal torus in $\cg$.

Let $Z=Z(G)$ be the center of $G$ and $Z(G)^0$ the identity component of $Z$.\\
Let $S$ be a maximal $\F$-split torus of $G$.\\
Let $S_\Delta := S\cap Z(G)^0$, the maximal $\F$-split torus in $Z(G)$.\\
The Weyl group of $S$ is denoted by $W := N_G(S)/Z_G(S)$.\\
For $\omega\subset G$, define $^G\omega := \{ gwg^{-1} : w\in \omega, g\in G\}$.\\
The root system of $(G,S)$ is denoted by $\Phi$. Let $\Phi^+$ be a system of positive roots and $\Delta$ the simple roots of $\Phi^+$. Define on $\Phi^+$ the height function $ht : \Phi^+\rightarrow \N$ as usual:
\[
\begin{array}{rll}
 ht(\alpha)&=1 &\alpha\in\Delta\\
 ht(\alpha+\beta) &= ht(\alpha)+ht(\beta) &\alpha,\beta,\alpha+\beta\in \Phi^+
\end{array}
\]

Let $ht(\Phi) := \max_{\alpha\in\Phi^+} ht(\alpha)$. Define $U^+ := \Pi_{\alpha\in \Phi^+} U_\alpha$ and $U^- :=  \Pi_{\alpha\in \Phi^-} U_\alpha$.\\

Let $\gamma$ be a regular semi-simple element. Let $\E$ be a field extension of $\F$ such that $T := Z_G(\gamma)$ is $\E$-split. Extend the valuation $v$ of $\F$ to $\E$. Let $\tilde{\Phi} := R(\cg , \ct)$. Define the singular depth of $\gamma$ as follows:
\[sd(\gamma) := \max_{\alpha\in \tilde{\Phi}} v(\alpha(t)-1).\]

\section{The extended and the reduced building}
In this section we give a construction of the reduced and the extended building.\\
For any reductive $p$-adic group Bruhat and Tits constructed a reduced and an extended building in \cite{BT72,BT84,Ti79}. We use the notation of \cite{MS12}.
The construction of the building goes as follows:
\begin{enumerate*}
\item we construct the standard apartment: a vector space $\A$ with a $N_G(S)$-action
\item we define for each vector $x\in \A$ a subgroup $U_x<G$
\item the building will be $B(G) := G\times \A / \sim$ where $\sim$ is a equivalence relation defined by:
$$(g,x)\sim (h,y) \Leftrightarrow \exists{n\in N_G(S)}\; [nx=y \text{ and } g^{-1}hn\in U_x].$$
\end{enumerate*}
Let $\A_e := X_*(S)\otimes_\Z \R$ and $\A_a := \left(X_*(S)/X_*(S_\Delta)\right)\otimes_\Z \R.$\\
Define $v:Z_G(S)\rightarrow X_*(S)\otimes_\Z \R$ by:
$$\left< v(z),\chi|_S\right> = -v(\chi(z))$$
for all $\chi\in X^*(Z_G(S))$.\\
Let $z\in Z_G(S)$ act on $\A_e$ by $x\mapsto x+v(z)$.\\
This action can now be extended to $N_G(S)$, $v : N_G(S)\rightarrow \text{Aff}(\A_e)$.

The standard apartment of the extended building is $\A_e$ and $\A_a$ is the standard apartment of the reduced building.\\
Define the linear map $\phi : \A_e\rightarrow \A_a$ by extending the map $X_*(S)\rightarrow X_*(S)/X_*(S_\Delta)$. So $\phi$ is surjective.\\
The action of $N_G(S)$ on $\A_e$ gives an action on $\A_a$ via $\phi$:
 $$n\phi(x) := \phi(nx), \quad \text{for } x\in\A_e,\;n\in N_G(S).$$

For $\alpha\in\Phi$, $x\in \A_a$ and $y\in \A_e$, define
\begin{align*}
 \alpha(y) &:= \left<\alpha,y\right>,\\
 \alpha(x) & := \alpha(z),
\end{align*}
where $z$ is any element in $\phi^{-1}(x)$.\\

Now we continue by defining the subgroups $U_x$.\\
Following \cite{Ti79} we construct subgroups $U_{\alpha,r}$ for $\alpha\in\Phi$ and $r\in \R$. Let $r_\alpha$ be the reflection associated to $\alpha$. Let $u\in U_\alpha-\{1\}$, then \[U_{-\alpha}uU_{-\alpha} \cap N_G(S) = \{m(u)\}.\] Define $r(u) = v(m(u))$. The affine action $r(u)$ is an affine reflection whose vector part is $r_\alpha$. Let $a(\alpha,u)$ denote the affine function on $\A_a$ whose vector part is $\alpha$ and whose vanishing hyper-plane is the fixed point set of $r(u)$. We define
$$ U_{\alpha,r} := \{ u \in U_\alpha \mid u=1\text{ or } a(\alpha,u)\geq \alpha+r\}.$$
In \cite[\S3]{MS12} a more concrete description of the groups $U_{\alpha,r}$ is given:\\
Let $\E$ be a field extension of $\F$ such that $\cg$ is $\E$-split. Extend the valuation $v$ of $\F$ to $\E$. Let $\ct$ be a maximal $\E$-split torus that contains $\cs$. Define $\Phi_\ct := R(\cg,\ct)$. Choose a Chevalley basis on $\gl$, the Lie algebra of $\cg$. Such a basis gives rise to an isomorphism $u_\beta : \E\rightarrow \cu_\beta(\E)$ for all $\beta\in \Phi_\ct$. Define $U_{\beta,r} := v^{-1}([r,\infty))$ for $\beta\in\Phi_\ct$. Let $\rho : \Phi_\ct \rightarrow \Phi$ be the surjection defined by restriction of the character of $\ct$ to $\cs$. For $\alpha\in\Phi^{\text{red}}$ and $r\in\R$ define
\begin{align*}
 U_{\alpha,r} &:= U_\alpha \cap \left(\prod_{\beta\in\rho^{-1}(\alpha)} U_{\beta,r}\times \prod_{\beta\in\rho^{-1}(2\alpha)}U_{\beta,2r}\right),\\
 U_{2\alpha,r} &:= U_{2\alpha}\cap U_{\alpha,r/2}.
\end{align*}

Now $U_x$, for $x\in \A_a$ or $x\in\A_e$, is the subgroup generated by $\bigcup_{\alpha\in\Phi}U_{\alpha,\left<x,-\alpha\right>}$.\\

As announced $B_a(G) := G\times \A_a/\sim$ and $B_e(G) := G\times \A_e/\sim$.\\
The equivalence relation $\sim$ for $B_a(G)$ and $B_e(G)$ is:\\
$(g,x)\sim (h,y)$ iff there is a $n\in N_G(S)$ such that $nx=y$ and $g^{-1}hn\in U_x$.\\

If $\Omega \subset \A_a$ or $\Omega\subset \A_e$ we define 
$$f_\Omega: \Phi \rightarrow \R\cup\{\infty\},\quad f_\Omega(\alpha) := \sup_{x\in\Omega} \left<x,-\alpha\right>. $$
This gives rise to the following subgroups of $G$:
\begin{align*}
 U_\Omega & := \text{ the subgroup generated by } \bigcup_{\alpha\in\Phi^{red}}U_{\alpha,f_\Omega(\alpha)},\\
 N_\Omega & := \{ n \in N_G(S) \mid nx=x \text{ for all } x\in \Omega\},\\
 P_\Omega &:= N_\Omega U_\Omega= U_\Omega N_\Omega.
\end{align*}
The group $P_\Omega$ is the point-wise stabilizer of $\Omega$.\\
If we drop $G$ from the notation of the building it should be clear from the context for which group $G$ the building is: so $B_a = B_a(G)$ and $B_e = B_e(G)$.\\

Now we extend $\phi:\A_e\rightarrow \A_a$ to a function $B_e\rightarrow B_a$ which we also denote by $\phi$. So $\phi(g,x)=(g,\phi(x))$. The function $\phi$ is $G$-invariant and surjective.\\

Let $Y := X_*(S_\Delta)\otimes \R= X_*(Z(G))\otimes \R$. Define $\pi_Y$ and $\pi_{\A_a}$ to be the projections from $\A_a\oplus Y$ to $Y$ and $\A_a$ respectively. Now we have a bijection $\Pi: \A_e\rightarrow \A_a\oplus Y$, such that $\phi = \pi_{\A_e}\Pi$. For $x\in \A_a$ and $y\in Y$ we have $x\oplus y\in \A_a\oplus Y$. We write $(x,y) := \Pi^{-1}(x\oplus y)$ or $x+y = \Pi^{-1}(x\oplus y)$.\\

For $\alpha\in \Phi$ define $n_\alpha$ to be the smallest $r\in \R^+$ such that $U_{\alpha,r} \not= U_{\alpha,r+}$. For $r\in \R$ define the $\alpha$-ceiling as: $\lceil r\rceil_\alpha := \min \{z\in n_\alpha\Z\mid z\geq r\}$. Let $\A$ be equal to $\A_e$ or $\A_a$. The affine hyperplanes
$$\A_{\alpha,k} := \{x\in \A \mid \left<x,\alpha\right>=k\}\quad \text{for }\alpha\in \Phi \text{ and } k\in n_\alpha\Z $$
turn $\A_a$ into a polysimplicial complex. An element $x\in \A_a$ is a vertex if it is the only element of an intersection of such affine hyperplanes. The polysimplicial vertices in $\A_e$ are $(\dim Z(G))$-dimensional hyperplanes. We call $x\in \A_e$ a vertex if it is an element of a polysimplicial vertex of $\A_e$. An element $x\in \A_e$ is a vertex if and only if $\phi(x)$ is a vertex.\\

For each $\Omega \subset B_a$ that is contained in an apartment and each $e\in \R_{\geq 0}$ Schneider and Stuhler defined a group $U_{\Omega}^e$ in \cite{SS97}. This group has the following properties.\\
For a point $x$, a polysimplex $\sigma$ and a general subset $\Omega$ of a apartment, the following hold:
\begin{enumerate}
 \item $U^e_\Omega$ is open if $\Omega$ is bounded.
 \item $U_\Omega^e$ is compact and normal in $P_\Omega$.
 \item if $e\in\Z_{\geq 0}$ and $x$ is in the interior of $\sigma$, then $U_x^e=U_\sigma^e$.
 \item $U_\Omega^e\subset U_\Omega^{e'}$ whenever $e\geq e'$.
 \item $U_\sigma^e$ for $e\in\N$ form a neighborhood basis of $1$ in $G$.
\end{enumerate}
This is a part of \cite[Theorem 5.5]{MS12}.\\
By definition a smooth representation has level greater or equal to $e\in\R_{\geq0}$ if $V=\sum_{x\in B_a} V^{U^{e}_x}$.

Define the bilinear symmetric form $\left<\;,\;\right>$ on $\A_a$ as follows:
$$\left<v,w\right> = \sum_{\alpha\in R^+} \alpha(v)\alpha(w).$$
This form is $W$-invariant, since $s_\alpha R^+ = R^+-{\alpha}\cup \{-\alpha\}$ for all $\alpha\in \Delta$. Let $\Delta^\vee$ be the dual basis of $\Delta$ in $\A_a$. Let $v\in \A_a-\{0\}$. If $v = \sum_{\alpha^\vee \Delta^\vee} c_{\alpha^\vee}\alpha^\vee$ with $c_{\alpha^\vee}\geq 0$, then $\left<v,v\right>>0$. Because the form is $W$-invariant, it is positive definite. Thus $\left<\;,\;\right>$ is a $W$-invariant positive definite symmetric bilinear form.

Choose on $\A_e = \A_a\oplus Y$ a $W$-invariant inner product, such that restricted to $\A_a$ it is equal to $\left<\;,\;\right>$ and $\A_a\perp Y$. Such an inner product exists, because $\A_a$ and $Y$ are $W$-invariant subspaces. This inner product gives rise to a $G$-invariant metric $d$ on $B_e$. Since $\A_a \perp Y$ one has
\[d( x + y , x'+ y') = (d(x,x')^2+d(y,y')^2)^{\frac12}.\]
\section{$\gamma$-fixed points and $D(\gamma)$}
An element $g\in G$ is called compact if and only if it is contained is a subgroup $K$ that is compact modulo $Z(G)$.\\
This section gives a proof of the following Theorem.
\begin{ste}\label{ugrens}
 Let $x,y\in \A_a$ and $\gamma \in Z_G(S)$ be regular and compact, then
 $$\#\{ ux : u\in U^+\cap P_y \mid \gamma ux=ux\} \leq |D(\gamma)|^{-\frac12}.$$
\end{ste}
In the proof of Theorem \ref{ugrens} we need Lemma \ref{ualem}. Besides some standard facts the proof of Theorem \ref{ugrens} uses only this Lemma, which is trivial when $G$ is $\F$-split. The main part of this section is dedicated to the proof of Lemma \ref{ualem}.\\

First we will discuss some consequences of the following Theorem.\\
Define $\cg_a := (k,+)$ and $G_a := \cg_a(\F)$.
\begin{ste}\label{pro14311}
 Let $\cg$ be a $\F$-split solvable group, $\ct$ be a maximal $\F$-torus of $\cg$ and $\cg_u$ the unipotent radical of $\cg$.
\begin{enumerate}[label=(\alph*)]
 \item There exists a $\F$-isomorphism of varieties $\fhs: \cg_u\rightarrow \cg_a^n$ with $\fhs(e)=0$ and a rational representation $\rho$ of $\ct$ in $k^n$ defined over $\F$ such that\\ $\fhs(tgt^{-1})=\rho(t)\fhs(g)$ for all $g\in \cg$ and $t\in \ct$.
 \item For $x,y\in \cg_a^n$ we have $\fhs(\fhs^{-1}(x)\fhs^{-1}(y))=x+y+\sum_{i\geq 2} F_i(x,y)$, where $F_i : \cg_a^n\times \cg_a^n \rightarrow \cg_a^n$ is a polynomial map of degree $i$.
 \item The weights of $\ct$ for $\rho$ are the weights of $\ct$ in $\gl$.
\end{enumerate}
\end{ste}
\begin{proof} See Proposition 14.3.11 in \cite{Sp98}. \end{proof}
\begin{gev}\label{visg}
 Let $\cs$ be a $\F$-split torus, $\cu$ be an $\F$-split unipotent group with an algebraic action of $\cs$. Let $n=\dim \cu$. Assume that $\alpha\in X^*(\cs)$ is the only weight for $\cs$ on $\ul$ and that $\alpha$ is non-trivial. Then there is a $\F$-isomorphism $\fhs$ between the groups $\cu$ and $\cg_a^n$ such that $\fhs(sus^{-1})=\alpha(s)\fhs(u)$ for all $s\in \cs$.
\end{gev}
\begin{proof} Apply Theorem \ref{pro14311} to $\cg=\cs\ltimes \cu$.\\
Let $\fhs : \cu \rightarrow \cg_a^n$ be an $\F$-isomorphism as in Theorem \ref{pro14311}. Then 
$$\fhs(\fhs^{-1}(x)\fhs^{-1}(y)) = x+y+\sum_{i\geq 2}F_i(x,y),$$
where $F_i(x,y): \cg_a^n\times \cg_a^n\rightarrow \cg_a^n$ is a polynomial map of degree $i$.\\
The weights of $\cs$ for $\rho$ are the weights of $\cs$ in $\gl$.\\
Since the only weight of $\cs$ in $\ul$ is $\alpha$, the weight of $\cs$ for $\rho$ is $\alpha$. Therefore $\rho(s)=\alpha(s)$ for all $s\in \cs$. Also
\begin{align*}
 \fhs(s\fhs^{-1}(x)\fhs^{-1}(y)s^{-1}) &= \rho(s) (x+y+\sum_{i\geq 2}F_i(x,y))\\
\fhs(s\fhs^{-1}(x)s^{-1}s\fhs^{-1}(y)s^{-1})&= \fhs(\fhs^{-1}(\rho(s)x)\fhs^{-1}(\rho(s)y))\\
& = \rho(s)x+\rho(s)y+\sum_{i\geq 2}F_i(\rho(s)x,\rho(s)y).
\end{align*}
Since $im(\rho) \cong k^\times$ and $k$ is infinite, $x+y+\sum_{i\geq 2} F_i(x,y)$ is a homogeneous polynomial map of degree $1$. Therefore $\fhs(\fhs^{-1}(x)\fhs^{-1}(y))=x+y$. So $\fhs$ is a group homomorphism between $\cu$ and $\cg_a^n$.\end{proof}
\begin{lem}
 Let $\cs$ be a maximal $\F$-split torus of the reductive group $\cg$. Let $\ct\subset Z_\cg(\cs)$ be a maximal $\F$-torus, $\alpha\in R(\cg,\cs)$ and $\cu_\alpha$ the unipotent group for $\alpha$. There are group isomorphisms $\fhs_1 : \cu_{\alpha}/\cu_{2\alpha} \rightarrow \cg_a^m$ and $\fhs_2 : \cu_{2\alpha} \rightarrow \cg_a^n$ such that for all $r\geq 0$:
\begin{enumerate}[label=(\alph*)]
  \item $\fhs_1(U_{\alpha,r}/U_{2\alpha,r})$ is an $\co$-lattice in $U_\alpha/U_{2\alpha}$,
  \item $\fhs_2(U_{2\alpha,r})$ is an $\co$-lattice in $U_\alpha$,
  \item The conjugation action of $\ct$ on $\cu_{\alpha}/\cu_{2\alpha}$ (resp. $\cu_{2\alpha}$) gives rise to a rational linear action $\rho_1$ (resp. $\rho_2$) of $\ct$ on $\fhs_1(\cu_{\alpha}/\cu_{2\alpha})$ (resp.  $\fhs_2(\cu_{2\alpha})$). Moreover the weights of $\ct$ for $\rho_1$ (resp. $\rho_2$) are the weights of $\ct$ in $\ul_{\alpha}/\ul_{2\alpha})$ (resp. $\ul_{2\alpha}$).
\end{enumerate}
\end{lem}
\begin{proof} We will only consider the case with $\fhs_1 : \cu_\alpha/\cu_{2\alpha,2r}\rightarrow \cg_a^m$. The proof with $\fhs_2$ goes analogously.\\

Let $\E$ be a finite field extension of $\F$ such that $\ct$ is $\E$-split. The group $\cu_\alpha$ is stable under conjugation with $\ct$, since $\ct\subset Z_\cg(\cs)$. Define $\fhs_\cs : \cu_\alpha/\cu_{2\alpha} \rightarrow \cg_a^m$ to be a $\F$-group isomorphism as in Corollary \ref{visg}. Let $\{\beta_1, \ldots,\beta_m\}$ be the subset of the roots of $\cg$ relative to $\ct$ such that $\beta_i|_\cs=\alpha$. Define $\fhs_\ct : \cu_\alpha/\cu_{2\alpha} \rightarrow \cg_a^m$ by its inverse: $\fhs_\ct^{-1}(x_1,\ldots,x_m) := \prod_{i=1}^m u_{\beta_i}(x_i) \mod \cu_{2\alpha}$ where the $u_{\beta_i} : \cg_a\rightarrow \cu_\beta$ are chosen in such a way that 
$$U_{\alpha,r}/U_{2\alpha,r} := \{ \prod_{i=1}^m u_{\beta_i}(x_i) \mod U_{2\alpha}: v(x_i)\geq r \} \cap U_\alpha/U_{2\alpha}.$$ The map $\fhs_\cs\fhs_\ct^{-1} : \cg_a^m\rightarrow \cg_a^m$ is an $\E$-group isomorphism. Because $\fhs_\cs\fhs_\ct^{-1}$ preserves the action of $\cs$, it is also an $\E$-linear map. Therefore there is a $\F$-structure on $\cg_a^m$ (in the sense of vector spaces) such that $\fhs_\ct$ is an $\F$-isomorphism between $\cg_a^n(\F)$ and $U_\alpha$. The group $\fhs_\ct(U_{\alpha,r}(\E)/U_{2a,r}(\E))$ is an $\co_\E$-lattice. So $$\fhs_\ct(U_{\alpha,r}/U_{2\alpha,r})= (\fhs_\ct(U_{\alpha,r}(\E)/U_{2\alpha,r}(\E))\cap \cg_a^m(\F)$$
is an $\co$-lattice.\\
The rank of the $\co$-lattice is $m$:\\
For all $x\in \E, \beta_i$ and $r\in \R$ one has
$$u_{\beta_i}(x)\in U_{\beta_i,r} \Leftrightarrow u_{\beta_i}(\pi x) \in U_{\beta_i,r+1}.$$
Since multiplication with $\pi$ respects the $\F$-structure on $\cg_a^m(\F)$ one has \[[U_{\alpha,r}U_{2\alpha}/U_{2\alpha} : U_{\alpha,r+1}U_{2\alpha}/U_{2\alpha}] = q^l,\] where $l$ is the rank of the $\co$-lattice $\fhs_1(U_{\alpha,r}/U_{2\alpha})$.\\
For all $\beta_i$ one has $\bigcup_{r\in\R} U_{\beta_i,r} = U_{\beta_i}$ and $U_{\beta_i,r} \leq U_{\beta_i,s}$ whenever $s \leq r$. Therefore also $\bigcup_{r\in \R} U_{\alpha,r}U_{2\alpha}/U_{2\alpha} = U_\alpha/U_{2\alpha}$. Because the rank of $U_{\alpha,r}$ is the same for all $r\in \R$ one has that the rank of $U_{\alpha,r}$ is $m$.\\
By construction of $\fhs_\ct$ the weights of $\rho_1$ are the same as the weights of the conjugation action of $\ct$ on $\ul_{\alpha}/\ul_{2\alpha}$. \end{proof}
\begin{lem}\label{lattices}
 Let $L'<L$ be $\co$-lattices in $\F^n$ (of rank $n$) and $M\in GL_n(\F)$ such that $ML'<L'$ and $ML < L$. Let $v\in L$, then
$$\#\{ l\in L/L' : Ml+v \in L'\} \leq |\det M|^{-1}.$$
\end{lem}
\begin{proof} We may assume that there exists at least one $l\in L/L'$ such that $Ml+v\in L'$. Take $n\in \N$ such that $L<\pi^{-n} L'$. Then
\begin{align*}
 \#\{ l\in L/L' : Ml+v \in L'\} &= \#\{l\in L/L' : Ml\in L'\}\\
 &\leq \#\{ l\in \pi^{-n} L'/L' : Ml \in L'\}.
\end{align*}
We will now estimate the last number.\\
Take a basis for $L'$.\\
Let $D$ be the Smith-normal form of $M$ with respect to $L'$, i.e. there are $P,Q\in GL(L')$ such that $PMQ=D$ and $D$ is a diagonal form. Then
\begin{align*}
 \#\{ l\in \pi^{-n} L'/L': Ml \in L'\} &= \#\{ l\in \pi^{-n} L'/L' : Dl \in L'\}\\
 &\leq |\det D|^{-1} = |\det M|^{-1}.
\end{align*}
The inequality follows from the fact that for all $c\in \F^\times$, the number of $a\in \co/\pi^n\co$ such that $ca\equiv 0 \mod \pi^n$ is bounded by $q^{v(c)} = |c|^{-1}$. \end{proof}
For $q,r\in \R$, define $exp_q(r) := q^r$. Recall that $T$ is a maximal torus of $G$ containing a maximal split torus $S$.
\begin{lem}\label{ualem}
 Let $t\in T$ be compact.\\
 Let $r,s\in \R$ and $r<s$.\\
 Let $V$ be a set of representatives for the cosets of $U_{2\alpha,s}$ in $U_{2a,r}$ and $U$ be a set of representatives for the cosets of $U_{\alpha,s}U_{2\alpha,r}$ in $U_{\alpha,r}$.
 \begin{enumerate}[label=(\alph*)]
 \item $\{ uv : u\in U,\;v\in V\}$ is a set of representatives for the cosets of $U_{\alpha,s}$ in $U_{\alpha,r}$.
 \item For $w,w'\in U_{\alpha,r}$ one has
 \end{enumerate}
 \[\#\{ (u,v) \in U\times V  \mid w'[(uv)^{-1},t]w \in U_{\alpha,s} \} \leq exp_q\left(\sum_{\beta\in \rho^{-1}\{\alpha,2\alpha\}} v(\beta(t)-1) \right).\]
\end{lem}
\begin{proof} Define $\fhs : U_\alpha \rightarrow U_\alpha/U_{2\alpha}$ to be the quotient map.\\
We first prove the following:
$$\#\{ u \in U  \mid \fhs(w'[u^{-1},t]w) \in \fhs(U_{\alpha,s}) \} \leq exp_q\left(\sum_{\beta\in \rho^{-1}(\alpha)} v(\beta(t)-1) \right).$$
The set $\fhs(U_{\alpha,s})$ is an $\co$-lattice and $\fhs(u)\mapsto \fhs([u,t])$ is a linear action on the lattice. Since this action has determinant $\prod_{\beta\in \rho^{-1}(\alpha)} (\beta(t)-1)$, the inequality follows from Lemma \ref{lattices}.\\

Now we prove that for every $u\in U$:
$$\#\{ v \in V  \mid w'[((uv)^{-1},t]w \in U_{\alpha,s} \} \leq exp_q\left(\sum_{\beta\in \rho^{-1}(2\alpha)} v(\beta(t)-1) \right).$$

We may assume that $\fhs(w'[u^{-1},t]w)\in \fhs(U_{\alpha,s})$.\\
So $w'[u^{-1},t]w = u_\alpha v'$ with $u_\alpha\in U_{\alpha,s}$ and $v'\in U_{2\alpha}$. Since $U_{2\alpha}$ is in the center of $U_{\alpha}$ and stable under conjugation with $t$, one has
$$w'[(uv)^{-1},t]w = w'[u^{-1},t]w [v^{-1},t] =u_\alpha v'[v^{-1},t].$$
The latter is in $U_{\alpha,s}$ if and only if $v'[v^{-1},t]\in U_{2\alpha,s}$. So
$$\{ v \in V  \mid w'[((uv)^{-1},t]w \in U_{\alpha,s} \}=\{ v \in V  \mid v'[v^{-1},t]\in U_{2\alpha,s}\}. $$
One gets the upper-bound for the number of $v$'s in the last set in the same way as in the case with $U$.\\
Combining both upper-bounds results in the upper-bound of the Lemma.
\end{proof}

{\it \noindent Proof of Theorem \ref{ugrens}.} Since $\gamma$ is compact, it fixes $\A_a$ pointwise.\\
Let $\Phi^+ = \{\alpha_1,\ldots , \alpha_k\}$ be the positive roots associated with $U^+$. Write $u\in U^+\cap P_y$ as $u = \prod_{i=1}^k u_{\alpha_i}$. Now $\gamma ux= ux$ if and only if $[u^{-1},\gamma]_{\alpha_i} \in U_{\alpha_i, -\alpha_i(x)}$ for all $i$. We will count the number of fixed points in the orbit of $x$ under $U^+\cap P_y$.\\
Let $R_\alpha$ be a set of representatives of the cosets $U_{y,-\alpha(y)}/U_{x,-\alpha(x)}$ for each $\alpha\in\Phi^+$.\\
We use the following bijection between $\prod_{\alpha\in \Phi^+} R_\alpha$ and $\{ ux : u\in U_y^+\}$:
$$(u_\alpha)_{\alpha\in\Phi^+} \mapsto \left(\prod_{\alpha\in\Phi^+}u_\alpha\right) x .$$

Let $v\in U^+_y$.\\
CLAIM: The number of $v_\beta\in U_{\beta,-\beta(y)}/U_{\beta,-\beta(x)}$ such that there is a $u$ with $u_\alpha = v_\alpha$ for the roots $\alpha$ with $ht(\alpha)<ht(\beta)$, $u_\beta = v_\beta$ and $\gamma ux=ux$ is bounded by $exp_q\left(\sum_{\tilde{\beta}\in \rho^{-1}(b),\;\tilde{\beta}\in \rho^{-1}(2b)} v(\tilde{\beta}(t)-1)\right) $.\\

If $\gamma ux=ux$, then $[u^{-1},\gamma]_\beta \in U_{\beta,-\beta(x)}$. Now $[u^{-1},\gamma]_\beta = w[u^{-1},\gamma]w'$, where $w,w'\in U_\beta$ only depend on the $u_\alpha$ with $ht(\alpha)<ht(\beta)$. Hence by Lemma \ref{ualem} the number of $v_\beta$'s is bounded by $exp_q\left(\sum_{\tilde{\beta}\in \rho^{-1}(\beta),\;\tilde{\beta}\in \rho^{-1}(2\beta)} v(\tilde{\beta}(t)-1)\right) $.\\

The claim allows us to prove with induction on the height of the roots that $$|\{ u \in \prod_{\alpha\in\Phi^+} R_\alpha \mid \gamma ux=ux\}|\leq exp_q\left(\sum_{\beta\in \rho^{-1}(2\Phi^+), \beta\in\rho^{-1}(\Phi^+)} v(\beta(t)-1)\right).$$
Since $T/S$ is compact, $v(\alpha(t)-1)\geq 0$ for all $\alpha\in R(G,T)$ with $\alpha|_S=0$. 
Thus
\begin{multline*}
 |\{ ux : u\in U^+\cap P_y \mid \gamma ux=ux\}|=|\{ u \in \prod_{\alpha\in\Phi^+} R_\alpha \mid \gamma ux=ux\}|\\
 \leq exp_q\left(\sum_{\beta\in \rho^{-1}(2\Phi^+), \beta\in\rho^{-1}(\Phi^+)} v(\beta(t)-1)\right)
 \leq |D(\gamma)|^{-\frac12}
\end{multline*}
So $\{ ux : u\in U^+\cap P_y \mid \gamma ux=ux\}$ has at most $|D(\gamma)|^{-\frac12}$ points.\hfill $\square$
\section{An upper-bound for the character}
The first part of this section up to and including Theorem \ref{lctr} is essentially in \cite{MS10} and \cite{MS12}.\\
Let $(\rho,V)$ be an admissible $G$-representation of level $e$.\\
For a open compact subgroup $K$ of $G$ we denote $1_K$ for the indicator function of $K$ in $G$ and $\left<K\right> := \frac{1_K}{\text{vol}(K)}$.\\
Let $B$ be the reduced building of $G$ and $\A$ be the standard apartment of $S$ in $B$. Define $\oo$ to be the origin of $\A$. For a finite subcomplex $\Sigma \subset B$ and $g\in G$ define
\begin{align*}
u^e_\Sigma &:= \sum_{\sigma\in\Sigma} (-1)^{\dim \sigma}
\left<U_\sigma^e\right>,\\
\tau_\Sigma(g) &:= \sum_{\sigma\in\Sigma^g} (-1)^{\dim
\sigma}\epsilon_\sigma(g) tr(\rho(g),V^{U_\sigma^e}),
\end{align*}
where $\Sigma^g$ is the set of $g$-stable polysimplices in $\Sigma$ and $\epsilon_\sigma(g)$ is $1$ if $g$ preserves the orientation of $\sigma$ and $-1$ otherwise. For $r\in\R$ define
\begin{align*}
\A^{\alpha+}_r &:= \{ x\in \A \mid \alpha(x)>r\},\\
\A^{\alpha0}_r &:= \{ x\in \A \mid \alpha(x)\in [-r,r]\},\\
\A^{\alpha-}_r &:= \{ x\in \A \mid \alpha(x)<-r\}.
\end{align*}
For any map $\epsilon:\Phi \rightarrow \{+,0,-\}$ we write 
$$ \A^\epsilon_r := \bigcap_{\alpha\in\Phi}
\A^{\alpha\,\epsilon(\alpha)}_r.$$
Let $\A^b_r$ be the union of the bounded $\A^\epsilon_r$. Define $B_r := P_\oo\A^b_r$.
\begin{lem}\cite[Lemma 8.2]{MS12}\label{colarg}
 Let $r\in\Z_{\geq e}$ and let $\Sigma\subset B$ be any finite convex
subcomplex that contains $B_{r-e}$. Then
 $$\left<U^r_\oo\right> u_\Sigma^e = \left<U^r_\oo\right>u_{B_{r-e}}^e.$$
\end{lem}

Take $r$ such that $\A^b_1\subset \A^{U^r_\oo}$. For $n\in \N_{\geq 1}$, then $A^b_n \subset \A^{U^{nr}_\oo}$. Define $C_n := P_\oo \A^{U^{nr}_\oo} = B^{U^{nr}_\oo}$. Now $C_n$ is a finite $P_\oo$-invariant convex subcomplex containing $B_n$.

For $\Sigma\subset B$ define $\Sigma^0 := \{ v\in \Sigma \mid v \text{ is a vertex of } B\}$.
\begin{ste}\label{trnbuil}
 For each $f\in C_c^\infty(P_\oo)$ and finite $P_\oo$-invariant convex
subcomplex $\Sigma_0$ such that $\text{im}\; f \subset \sum_{x\in
\Sigma_0^0} V^{U_x^e}$ one has
 $$tr(f,V) = \int_{P_\oo} f(g)\tau_{\Sigma_0}(g)dg.$$
\end{ste}
\begin{proof} See the proofs in \cite[Theorem 2.4 and Proposition
4.1]{MS10}. \end{proof}

\begin{ste}\label{lctr}
 If $\gamma\in Z_G(S)$ is regular semi-simple, then 
 \begin{enumerate}[label=(\alph*)]
 \item $tr$ is constant on $\gamma U_\oo^{\max(sd(\gamma),e)}$,
 \item for all $\Sigma\subset B_a$, $\tau_\Sigma$ is constant on $U_x^{\max(ht(\Phi)sd(\gamma),e)}\gamma$.
 \end{enumerate}
\end{ste}
\begin{proof}
 See the proof of \cite[Theorem 7.2]{MS12}.
\end{proof}

\begin{gev}\label{trnpabuil}
 Let $\gamma \in P_\oo$ and $r\geq ht(\Phi)sd(\gamma)$, then 
$$tr_\rho(\gamma,V) = \tau_{C_{r-e}} (\gamma).$$
\end{gev}
\begin{proof}
Since $\gamma \in P_\oo$ and $U^r_\oo\nd P_\oo$, the
endomorphisms $\rho(\gamma)$ and $\rho(\left<U^r_\oo\right>)$ commute. Thus $im\;\rho(\gamma \left<U^r_\oo\right>)\subset V^{U^r_\oo}$. There exists a finite convex subcomplex $\Sigma$ containing $B_{r-e}$ such that $\left<U^r_\oo\right> u_\Sigma^e V = V^{U^r_\oo}$, because $U_\oo^r$ is an open compact group and $V$ admissible. So by Lemma \ref{colarg} $
\left<U^r_\oo\right>u_{B_{r-e}}^e V = V^{U^r_\oo}$. Since $B_{r-e}$ is $P_\oo$-invariant, the space $u_{B_{r-e}}^e V$ is $U^r_\oo$-invariant. Thus $V^{U_\oo^r} \subset u_{B_{r-e}}^e V$. Now $C_{r-e}$ is convex and contains $B_{r-e}$, so the requirements in Theorem \ref{trnbuil} are fulfilled for
$f=\gamma\left<U_\oo^r\right>$ and $\Sigma_0=C_{r-e}$. Therefore by Theorem \ref{lctr}
\[tr(\gamma,V)=tr\left(\gamma\left<U^r_\oo\right>V\right)=\frac{1}{\text{vol}(U^r_\oo)}\int_{U^r_\oo}\tau_{C_{r-e}}(\gamma g)dg=\tau_{C_{r-e}}(\gamma)\qedhere \]
\end{proof}

\begin{lem}\label{1lb}
 Let $h \in P_x$. There exists a $C$ such that for all $g\in h U_x^0$ and all simplices $\sigma\in B^g$:
 $$|tr(\rho(g),V^{U_\sigma^e})|\leq C.$$
\end{lem}
\begin{proof}
Denote $Z(G)$ with $Z$. Let $N$ be the order of the quotient group $P_x/(Z U_x^0)$. Take $z\in Z$ and $k\in U^0_x$ such that $h^N=zk$. Define $k' := g^Nz^{-1}$, then $k'\in U_x^0$ and $g^N=zk'$.

Since $g$ and $z$ fix $\sigma$, so does $k'$. Hence $k'$ is in $U_x^0\cap P_\sigma$. Let $m := \dim V^{U_\sigma^e}$. Choose on $V^{U_\sigma^e}$ a basis such that $\rho(z)$ and $\rho(g)$ are upper triangular matrices. Now also $\rho(k')$ is an upper triangular matrix. Let $\kappa_1,\ldots,\kappa_m$, $\lambda_1,\ldots, \lambda_m$ and $\nu_1,\ldots,\nu_m$ be the entries on the diagonal of the matrices $\rho(g)$, $\rho(z)$ and $\rho(k')$, respectively. Define $c(z,\sigma):= \sum_{i=1}^m |\lambda_i|^\frac{1}{N}$. Since $k'$ is contained in a compact subgroup acting on $V^{U_\sigma^e}$, $|\nu_i|=1$. Thus $|\kappa_i^N|=|\lambda_i\nu_i|=|\lambda_i|$. Hence $|tr(\rho(g),V^{U_\sigma^e})|\leq \sum_{i=1}^m|\lambda_i|^\frac{1}{N}=c(z,\sigma)$.\\
Because $z$ is in the center of $G$, for all $\sigma$ and $\sigma'$ in the same $G$-orbit,\\ $c(z,\sigma)=c(z,\sigma')$. (The eigenvalues and their multiplicity for $\rho(z)$ on $V^{U_\sigma^e}$ and $V^{U_{\sigma'}^e}$ are the same.) Since there are only finitely many $G$-orbits of simplices in $B$, there is a $C_z$ such that $c(z,\sigma)\leq C_z$ for all simplices $\sigma\in B$. Thus $|tr(\rho(g),V^{U_\sigma^e})|\leq C_z$ for all $\sigma\in B^g$. \end{proof}

Recall that for regular semi-simple elements $\gamma$, 
$$D(\gamma) := \prod_{\alpha \in R(G,Z_G(\gamma))} (\alpha(\gamma)-1).$$
Define $n := \dim \A_a$.
\begin{pro}\label{afschatting1}
 Let $g\in P_x$. There exists a $C\in\R$ depending only on the affine building of $G$, the element $g$ and the representation $(\rho,V)$, such that for all semi-simple regular $\gamma \in\; ^G{Z_G(S)}\cap gU_x^0$:
 $$|tr(\gamma,V)| \leq C(ht(\Phi)sd(\gamma)+1)^{n}|D(\gamma)|^{-\frac12}.$$
\end{pro}
\begin{proof}
Take a $c_b\in\R$ depending on the affine building such that for each $r\in \N$ the number of simplices in $C_{r}\cap \A$ is bounded by $c_br^n$.

Let $h\in G$ be such that $\gamma \in Z_G(hSh^{-1})$. Combining Theorem \ref{lctr} and Corollary \ref{trnpabuil} results in $tr(\gamma,V) = \tau_{hC_{ht(\Phi)sd(\gamma)}}(\gamma)$.  The number of simplices in $hC_{ht(\Phi)sd(\gamma)}\cap h\A$ is bounded by $c_b (ht(\Phi)sd(\gamma)+1)^n$. By Theorem \ref{ugrens} the number of $\gamma$-fixed simplices in $hC_{ht(\Phi)sd(\gamma)}$ is bounded by $c_b (ht(\Phi)sd(\gamma)+1)^n |D(\gamma)|^{-\frac12}$. By Lemma \ref{1lb} $|tr(\rho(\gamma),V^{U^e_\sigma})|\leq C$ for all $\gamma\in gU^0_x$ and $\sigma\in B^\gamma$. Thus
\begin{align*}
tr(\gamma,V) &= \tau_{hC_{ht(\Phi)sd(\gamma)}}(\gamma)=\sum_{\sigma\in (hC_{ht(\Phi)sd(\gamma)})^\gamma}(-1)^{\dim \sigma}\epsilon_\sigma(\gamma) tr(\rho(g),V^{U^e_\sigma})\\ &\leq Cc_b(ht(\Phi)sd(\gamma)+1)^n |D(\gamma)|^{-\frac12}.\qedhere
\end{align*}
\end{proof}

Now we have an upper-bound for the trace of the compact regular elements in $Z_G(S)$ in a neighborhood of a compact element of $G$. For a general regular element in $Z_G(S)$ in a neighborhood of a general element of $G$ we use Casselman's method to compute the character.\\

Let $P$ be a $\F$-parabolic subgroup of $G$, $N$ its unipotent radical and $M$ a Levi factor of $P$. For a representation $(\rho,V)$ of $G$ define $$V(N) := \left< v-\pi(n)v: v\in V,n\in N\right>$$ and $V_N := V/V(N)$. Now $M$ acts on $V_N$ via the action of $M$ on $V$. The action of $M$ on $V_N$ is denoted by $\rho_M$. The $M$-module $(\rho_M,V_N)$ is called the Jacquet module of $V$.\\

For $g\in G$ we have the parabolic subgroup contracted by $g$: 
\begin{align*}
 P_g &:= \{ p\in G : \{g^npg^{-n} : n\in\N\} \text{ is bounded}\}\text{ and}\\
 M_g &:= \{ p\in G : \{g^npg^{-n} : n\in\Z\} \text{ is bounded}\}.
\end{align*}

By \cite[Proposition 2.3]{MS12} $P_g$ is a parabolic subgroup of $G$, $M_g$ is a Levi subgroup and $g$ is, viewed as element of $M_g$, compact. Roughly speaking by looking at the group $M_g$, the center of the group containing $g$ is enlarged in such a way that $g$ is compact modulo this enlarged center.

\begin{mdef}
 Let $g\in G$. We define the displacement function $d_g: B_e\rightarrow \R$ by $d_g(x) := d(x,gx)$. Let $d(g) := \inf_{x\in B_e} d_g(x)$.
\end{mdef}

Let $l$ be a line contained in $\A_e$. Let $\Phi_l$ be the set of roots $\alpha$ of $S$ such that $\left<\alpha,\cdot\right>$ is constant on $l$. Let $M_l$ be the Levi subgroup of $G$ generated by $Z_G(S)$ and the groups $U_\alpha$ for $\alpha\in \Phi_l$.

\begin{lem}\label{conmicong}
 Let $M$ be a Levi subgroup of a parabolic subgroup of $G$. Let $S$ be a maximal split torus in $M$.\\
 The regular semi-simple elements in $^GZ_G(S)\cap M$ are the regular semi-simple elements in $^{M}Z_{M}(S) \cap M$.
\end{lem}
\begin{proof}
If $\gamma\in\; ^GZ_G(S)\cap M$ and $\gamma$ is regular, then $Z_M(\gamma)$ is a maximal torus of $M$. Since the ranks of $G$ and $M$ are the same, $Z_M(\gamma)$ is also a maximal torus of $G$. Now $Z_G(\gamma)$ is a maximal torus of $G$, so $Z_G(\gamma)=Z_M(\gamma)$. Take $g\in G$ such that $\gamma \in Z_G(gSg^{-1})$. Thus $gSg^{-1}<Z_G(\gamma)=Z_M(\gamma)$. Since $gSg^{-1}$ is a maximal split torus of $G$, it is also a maximal split torus of $M$. Since $M$ is reductive, maximal split tori in $M$ are conjugated over $M$. So there is a $m\in M$ such that $gSg^{-1} = m S m^{-1}$. Thus $\gamma \in Z_M(mSm^{-1})$.
\end{proof}

The following is in the extended building.\\
For the moment let $g\in G$ be non-compact modulo the center. Thus $d(g)\not=0$. Let $l$ be a line in $B_e$ on which $g$ acts by translations. Such a line exists by \cite[Lemma 3.4.4]{DB02b}. Let $S$ be a maximal split-torus such that $l$ is in the apartment of $S$. By \cite[Lemma 3.4.4]{DB02b} $M_g=M_l$. So in particular $S\subset M_g$. Take $x\in l$, then $d_g(x)=d(g)$. Let $H = P_{[x,gx]}$. From the proof of \cite[Lemma 3.4.7]{DB02b} we see that
\begin{align}
d(gh)=d(g) \text{ for all $h\in H$.} \label{diseq}
\end{align}
\begin{lem}\label{conjugatievanleviofg}
 Let $h\in H$. The group $M_{gh}$ is conjugated with $M_g$ by an element of $H$.\\
 If moreover $gh\in M_g$, then $M_{gh}=M_g$.
\end{lem}
\begin{proof}
Since $h\in H$, $d(gh)=d(x,gx)=d(x,ghx)$. By the proof of \cite[Lemma 3.4.4]{DB02b} there is a line $l'$ such that the points $(gh)^nx$ for $n\in \Z$ are on $l'$. This line lies in an apartment $\A'$. Now $[x,gx] \subset l\cap l'$. By \cite[7.4.9]{BT72} there is a $h_o\in H$ such that $h_o\A=\A'$. In the apartment $\A$ (respectively $\A'$) there is only one way to continue the line segment $[x,gx]$ namely $l$ ($l'$ , respectively). Since $h_o$ fixes $[x,gx]$ and maps lines to lines, we have $h_ol=l'$. So $$M_{gh}=M_{l'} = M_{h_ol}=h_oM_lh_o^{-1}=h_oM_gh_o^{-1}.$$

Assume that $g':= gh\in M_g$. Since $g'$ fixes $x$ in $B_a(M_{g})$, $g'$ is compact modulo the center of $M_g$. Because $g'$ is compact modulo the center, if $m\in M_g$ then $\{ g'^{n}mg'^{-n} : n\in \Z\}$ is bounded. So $M_g \subset M_{g'}$. Thus $g\in M_{g'}$. Since $g$ fixes $x$ in $B_a(M_{g'})$, also $M_{g'}\subset M_g$. So $M_{gh}=M_{g'}=M_g$. 
\end{proof}
\begin{pro}\label{afschattingalgemeen}
 For every $g\in G$ there exists a constant $C\in\R$, such that for all semi-simple regular $\gamma \in \;^G{Z_G(S)}\cap gP_{[x,gx]}$:
 $$|tr(\gamma,V)| \leq C(ht(\Phi)sd(\gamma)+1)^{\dim \A_a}|D(\gamma)|^{-\frac12}.$$
\end{pro}
\begin{proof}
If $g$ is compact modulo the center we can use Proposition \ref{afschatting1}.\\ 
Assume that $g$ is not compact, then $d(g)\not=0$.\\
Let $H = P_{[x,gx]}$ for a $x\in B_e$ such that $d_g(x)=d(g)$. By conjugating $g$ we may assume that $g^z x\in\A_e$ for all $z\in\Z$. 

Let $N_g$ be the unipotent radical of $P_g$. Let $(V_{N_g},\rho_{N_g})$ be the Jacquet representation of $M_g$ for $\rho$. 

To indicate the difference between the objects defined for $G$ and $M_g$ the one corresponding with $M_g$ are labelled by $M_g$, e.g. $B_e$ is the building of $G$ and $B_e(M_g)$ is the building of $M_g$, $D_{M_g}$ is the Harish-Chandra function $D$ for $M_g$.

By \cite[Lemma 3.4.4]{DB02b} $\A_e$ is an apartment of $B_e(M_g)$ and the image of $x$ in $B_a(M_g)$ is a $g$-fixed point.\\

Let $\gamma$ be a semi-simple regular element in $^GZ_G(S)$.\\
Assume that $\gamma \in M_g$ and $\gamma\in gP_x$. By Lemma \ref{conmicong} $\gamma \in \;^{M_g}Z_{M_g}(S)\cap M_g$. Also $\gamma \in gP_x\cap M_g= gP_x(M_g)$. Let $P$ be a parabolic subgroup containing $M_g$ and let $N$ be the unipotent radical of $P$. By Proposition \ref{afschatting1} applied to $(\rho_{N},V_{N})$, there is a $C\in\R$ such that for all such $\gamma$ with $N=N_\gamma$,
$$|tr(\gamma,V_{N_\gamma})| \leq C(ht(\Phi)sd_{M_g}(\gamma)+1)^n|D_{M_g}(\gamma)|^{-\frac12}.$$
This $C$ can and will be chosen independently of $P$ and $N$, since there are only finitely many parabolic subgroups containing $M_g$.\\
By Casselman \cite{Ca77} $tr(\gamma,V)=tr(\gamma,V_{N_g})$. Thus for all $\gamma\in M_g$ with $\gamma \in gP_x$,
\begin{align}\label{ineqmg}
|tr(\gamma,V)| \leq C(ht(\Phi_{M_g})sd_{M_g}(\gamma)+1)^n|D_{M_g}(\gamma)|^{-\frac12}.
\end{align}

\begin{lem}\label{D/Dm}
 There exists a $C'\in\R_{>0}$ such that for all semi-simple regular elements $\gamma \in gP_x\cap M_g$ one has $\frac{|D(\gamma)|}{|D_{M_g}(\gamma)|}\leq C'$.
\end{lem}
\begin{proof}
We are going to construct a continuous function on $M_g$, which on the semi-simple regular elements $\gamma$ takes the value $\frac{D(\gamma)}{D_{M_g}(\gamma)}$.\\
Pick a basis $b_1,\ldots,b_n$ of $\gl$ such that $b_1,\ldots, b_{\dim M_g}$ is a basis for $\ml_g$. 
Let $g'\in M_g$. Write the matrix $ad(g')$ with respect to this basis. Let $\varphi(g)$ be determinant of the submatrix of $ad(g')$ in the lower right corner of dimension $\dim G-\dim M_g$. Then clearly $\varphi : M_g\mapsto \F$ is continuous.\\

Let $\gamma$ be a semi-simple regular element in $M_g$. Notice that the definition of $\varphi$ is independent of the choice of a basis with the property that the first $\dim M_g$ basis elements are in $\ml_g$. Since $\gamma$ is semi-simple regular, $T_\gamma := Z_G(\gamma)$ is a maximal torus and contained in $M_g$. Choose as basis for $\gl$ a basis for $\tl_\gamma$ and the eigenvectors $\mathfrak{u}_\alpha$, $\alpha\in R(M_g,T_\gamma)$, supplemented with the eigenvectors $\mathfrak{u}_\beta$ for all $\beta\in R(G,T_\gamma)-R(M_g,T_\gamma)$. So $$\varphi(\gamma) = \prod_{\beta\in R(G,T_\gamma)-R(M_g,T_\gamma)} 1-\beta(\gamma)=\frac{\prod_{\alpha\in R(G,T_\gamma)} 1-\alpha(\gamma)}{\prod_{\alpha\in R(M_g,T_\gamma)} 1-\alpha(\gamma)}=\frac{D(\gamma)}{D_{M_g}(\gamma)}.$$

Since $\varphi$ is continuous and $gP_x\cap M_g$ is compact, there is a $C'$ such that $\frac{|D(\gamma)|}{|D_{M_g}(\gamma)|}\leq C'$ for all semi-simple regular elements $\gamma \in gP_x\cap M_g$.
\end{proof}

{\noindent\it Continuation of the proof of Proposition \ref{afschattingalgemeen}.}  Combining Lemma \ref{D/Dm} with the estimate of the trace (\ref{ineqmg}), $sd_{M_g}(\gamma)\leq sd(\gamma)$ and $ht(\Phi_{M_g})\leq ht(\Phi)$ we get for all semi-simple regular $\gamma \in gP_x\cap M_g$:
\begin{align}
\begin{split}\label{ineqg}
|tr(\gamma,V)| &\leq C(ht(\Phi_{M_g})sd_{M_g}(\gamma)+1)^n|D_{M_g}(\gamma)|^{-\frac12}\\
&\leq C\sqrt{C'}(ht(\Phi)sd(\gamma)+1)^n |D(\gamma)|^{-\frac12}.
\end{split}
\end{align}

Assume that $\gamma\in ^G{Z_G(S)}\cap gP_{[x,gx]}$. There is, by Lemma \ref{conjugatievanleviofg}, a $h\in H$ such that $hM_\gamma h^{-1}= M_g$. Now $h\gamma h^{-1}\in M_g$ and $h\gamma h^{-1} \in g P_x$, because\\ $h\gamma h^{-1} x = h\gamma x = hgx=gx$. Thus by (\ref{ineqg}):
\begin{align*}
|tr(\gamma,V)| &= |tr(h\gamma h^{-1} ,V)|\leq  
 C\sqrt{C'}(ht(\Phi)sd(h\gamma h^{-1})+1)^n|D(h\gamma h^{-1})|^{-\frac12}\\
 &= C\sqrt{C'}(ht(\Phi)sd(\gamma)+1)^n|D(\gamma)|^{-\frac12}.\qedhere
\end{align*}
\end{proof}

\section{An estimate for the Weyl integration formula}
Let $T := Z_G(\gamma)$ be the maximal torus containing $\gamma$. Let $n := \dim \A_a$.\\
In this section we want to give an estimate of the Weyl integration formula. To be precise, we will show that for every $f\in C_c^\infty(G)$ there exists a $C\in \R$ such that for all semi-simple regular $\gamma \in Z_G(S)$ the following inequality holds:

$$|\int_{T\backslash G} f(g^{-1}\gamma g) dg| \leq C(ht(\tilde{\Phi})sd(\gamma)+1)^n|D(\gamma)|^{-\frac12}.$$

For $g\in G$ define 
$$B(g) := \{ x\in B_e(G) \mid d_g(x)=d(g)\}.$$

Let $g\in G$ and $x\in B(g)$. We will first give an estimate in the case that $f := 1_{gP_[x,gx]}$. Let $\gamma \in Z_G(S)\cap gP_{x,gx}$ be a semi-simple regular element. By equation (\ref{diseq}) $d(\gamma)=d(g)=d(\gamma x,x)$, so $x\in B(\gamma)$. For simplicity we estimate the integral of $1_{\gamma P_x}$ instead of $1_{gP_{[x,gx]}}$. Let $\phi_{M_\gamma} : B_e(M_\gamma)\rightarrow B_a(M_\gamma)$ be the canonical projection.\\
The relation between the integral and points in the building is due to the fact that if $1_{\gamma P_x}(g^{-1}\gamma g)=1$, then $gx\in B(\gamma)\subset B_e(M_\gamma)$, since
$$d(gx,\gamma gx)=d(gx,g\gamma x)=d(x,\gamma x)=d(\gamma).$$
So we need to identify the elements in $Gx\cap B(\gamma)$. Or more precisely, the $T$-orbits in $Gx\cap B(\gamma)$, because we are integrating over $T\backslash G$. To give an upper-bound for the number of $T$-orbits in $Gx\cap B(\gamma)$, we look at $B_{a,x}(\gamma)=\phi_{M_\gamma}(Gx\cap B(\gamma))$. Now $B_{a,x}(\gamma)$ consists of $\gamma$-fixed points. After some technicalities we get an upper-bound for the number of $T$-orbits of $\gamma$-fixed points. This upper-bound can certainly be improved, since it takes the measure on $T\backslash G$ into account.\\

Let $F$ be the fundamental domain of $T$ in $\A_e$ defined by 
$$F := \{ x\in \A_e \mid \forall{t\in T}\;[d(x,\oo)\leq d(x,t\oo)] \}.$$
\begin{mdef}
Let $x\in \A_a$ and $z\in B_a$, then $z$ is called above $x$ if $x$ is a vertex and $d(x,z)\leq d(v,z)$ for all vertices $v\in \A_a$.\\ 
Let $x+ y \in \A_e$ be a vertex, with $x\in \A_a$ and $y\in Y$. Let $z\in B_e$. Then $z$ is called above $x+ y$ and $x+ y$ is called below $z$ if $\phi(z)$ is above $\phi(x)$ and $d(z,x+ y)\leq d(z,x+ y')$ for all $y'\in Y$. 
\end{mdef}
\begin{lem}\label{y=y'}
 Let $x+ y, z+ y'\in \A_e$, with $x,z\in\A_a$ and $y,y'\in Y$. If $u (z+ y')$ is above $x+ y$ for $u\in U_x = U_{x+ y}$, then $y'=y$.
\end{lem}
\begin{proof}
Let $y''\in Y$. Since $u\in U_x=U_{x+ y''}$,
$$d( u(z+ y'),x+ y'') = d( z+ y', x+ y'').$$ Now $d (z+ y',x+ y'') = (d(z,x)^2+d(y',y'')^2)^{\frac12}$. Therefore $y=y'$.
\end{proof}

\begin{lem}\label{bermaat}
 Let $G$ and $H$ be unimodular groups such that $H$ is a closed subgroup of $G$. Let $K$ be an open compact subgroup of $G$. Suppose that the measures of $G$, $H\backslash G$ and $H$ are chosen in such a way that $\mu_H(H\cap K)=\mu_{H\backslash G}(HK)=\mu_G(K)=1$. Then, for any $g\in G$, \[\mu_{H\backslash G}(HgK) := \frac{[H\cap K: H\cap K \cap gKg^{-1}]}{[H\cap gKg^{-1} : H\cap K\cap gKg^{-1}]}.\]
\end{lem}
\begin{proof}
See \cite[II.3.9]{RE10} for a proof of the existence of a $G$-invariant measure on $H\backslash G$.
\[
 \int_H 1_{gK}(hg) dh = \int_H 1_{gKg^{-1}}(h)dh = \mu_H(gKg^{-1}\cap H).
\]
Thus $\int_H 1_{gK}(hg')dh = \mu_H(gKg^{-1}\cap H)1_{HgK}(g')$ for all $g'\in G$.\\
By the choice of the measure on $H$ we have:
\[
 \mu_H(gKg^{-1}\cap H) = \frac{[gKg^{-1}\cap H:K\cap gKg^{-1}\cap H]}{[K\cap H:K\cap gKg^{-1}\cap H]}.
\]
Since $\int_{H\backslash G} \int_H 1_{gK}(hx)dhdx = \int_G 1_{gK}(x)dx=\mu_G(gK)=1$,
\begin{align*}
\int_{H\backslash G} 1_{HgK}(x)dx &= \frac{1}{\mu_H(gKg^{-1}\cap H)}\int_{H\backslash G} \int_H 1_{gK}(hx)dhdx\\
&=\frac{[K\cap H:K\cap gKg^{-1}\cap H]}{[gKg^{-1}\cap H:K\cap gKg^{-1}\cap H]}. \qedhere
\end{align*}
\end{proof}

Let $K := P_x$. Take the measures on $G$, $T$ and $T\backslash G$ as in Lemma \ref{bermaat}. 
\begin{align*}
 L &:= \{ y\in B_e \mid y \text{ is above a vertex of } F\}.\\
 L_\gamma & := \{ y\in L \mid y\in Gx \text{ and } y\in \phi_{M_\gamma}^{-1}(B_a^\gamma)\}. 
\end{align*}
\begin{lem}\label{intnaarly}
 $\int_{T\backslash G} 1_{\gamma P_x}(g^{-1}\gamma g)dg \leq |L_\gamma|$.
\end{lem}
\begin{proof}
We will prove the following inequalities:
\begin{align}
 \int_{T\backslash G} 1_{\gamma P_x}(g^{-1}\gamma g)dg &= \sum_{g\in T\backslash G/P_{[x,\gamma x]}} 1_{P_x}(\gamma^{-1}g^{-1}\gamma g)\mu_{T\backslash G}(TgP_{[x,\gamma x]})\label{in1}\\
 &\leq \sum_{g\in T\backslash G/P_x} \mu_{T\backslash G}(TgP_x) 1_x(gx)\label{in2} \\
 &\leq \sum_{g\in (T\cap P_x)\backslash G/P_x \mid gx\in L_\gamma} \mu_{T\backslash G}(TgP_x) \label{in3}\\
 &\leq \sum_{g\in (T\cap P_x)\backslash G/P_x \mid gx\in L_\gamma} |(T\cap P_x)gx|\label{in4}\\
 &= |L_\gamma|.\label{in5}
\end{align}
Since for $g'\in TgP_{[x,\gamma x]}$ we have 
$$\gamma g x= g\gamma x \Leftrightarrow \gamma g' x= g'\gamma x,$$
the function $g\mapsto 1_{\gamma P_x}(g^{-1}\gamma g)$ is constant on double cosets $T\backslash G/P_{[x,\gamma x]}$. Therefore we have equality (\ref{in1}).\\
Define $1_x : B_e \rightarrow \R$ by $$1_x(y) := \left\{ \begin{array}{cc} 1 & \exists g\in G[ y=gx \wedge \gamma g x = g\gamma x]\\ 0 &\text{otherwise.}\end{array}\right.$$
Now $1_x(gx)=1$ if and only if there exists an $h\in gP_x$ such that $1_{P_x}(\gamma^{-1}h^{-1}\gamma h)=1$. Also $1_x(y)=1_x(ty)$ for all $t\in T$. So 
\begin{multline*}
\sum_{h\in T\backslash TgP_x/P_{[x,\gamma x]}}1_{P_x}(\gamma^{-1}h^{-1}\gamma h)\mu_{T\backslash G}(ThP_{[x,\gamma x]})\\
\leq \sum_{h\in T\backslash TgP_x/P_{[x,\gamma x]}}1_{x}( hx)\mu_{T\backslash G}(ThP_{[x,\gamma x]}) = 1_{x}(gx)\mu_{T\backslash G}(TgP_x).
\end{multline*}
This gives inequality (\ref{in2}).\\
For every coset $Tg$ there exists a $g'\in Tg$ such that $g'x\in L$. If moreover $1_x(gx)=1$, then $g'x\in B(\gamma)$. So $g'x\in L_\gamma$ and inequality (\ref{in3}) follows.\\
From Lemma \ref{bermaat} and $gP_xg^{-1}=P_{gx}$ we get inequality (\ref{in4}):
\begin{align*}
\mu_{T\backslash G}(TgP_x) &= \frac{[T\cap P_x : T\cap P_x\cap gP_xg^{-1}]}{[gP_xg^{-1}\cap T: T\cap P_x\cap gP_xg^{-1}]}\\&\leq [T\cap P_x : T\cap P_x\cap gP_xg^{-1}]=|(T\cap P_x) gx|.
\end{align*}

The group $T\cap P_x$ fixes $\A_e$ pointwise and commutes with $\gamma$, so it acts on $L_\gamma$. So the sum in (\ref{in4}) is over the $(T\cap P_x)$-orbits in $L_\gamma$. Each orbit contributes to the sum the number of elements in that orbit. Thus the sum is the number of elements in $L_\gamma$. Therefore equality (\ref{in5}) holds.
\end{proof}

\subsection{$\gamma$-fixed points in the reduced building}
In this subsection we assume that $\gamma \in Z_G(S)$ is a compact semi-simple regular element of $T$.\\
Define $\Phi := \Phi(G,S)$ and $\tilde{\Phi} := \Phi(\cg,\ct)$. Let $\rho : \tilde{\Phi}\rightarrow \Phi\cup\{0\}$ be the canonical projection. Define $n := \dim \A_a$.\\
The goal of this section is to prove the following Theorem:
\begin{ste}\label{Lgrens}
 There is a $c\in\R$ such that for all vertices $x\in \A_a$ and $\gamma \in T\cap P_\oo$ the following holds. The number of vertices fixed by $\gamma$ above $x$ is bounded by $c (ht(\tilde{\Phi})sd(\gamma)+1)^n|D(\gamma)|^{-\frac12}$.
\end{ste}

Let $C$ be a Weyl chamber of $\A_a$ with vertex $\oo$, $\mathcal{C}$ the cone of $C$, $\Delta = \{\alpha_1,\ldots,\alpha_n\}$ the set of simple roots associated to $\mathcal{C}$ and $\Phi^+$ the set of positive roots.\\
Define for each simple root $\alpha_i$ a vertex $a_i$ in $\A_a$ in the following way. Let $\Gamma\subset \Delta$ be the connected part of $\alpha_i$ in the Dynkin diagram. Let $\beta_0 := \sum_{\alpha_j\in\Gamma} c_j \alpha_j$ be the longest positive root in the root system generated by $\Gamma$. Define $a_i$ to be the vertex in $\A_a$ such that $\alpha_j(a_i)= \frac{\delta_{ij}}{c_i}$.

\begin{lem}\label{gwva}
  For all $i\in \{1,\ldots,n\}$ one has $d(\oo,x+ta_i)>d(\oo,x)$ for $t\in\R_{>0}$ and $x\in \mathcal{C}$.
\end{lem}
\begin{proof}
 Recall that $d(\oo,x)=\left<x,x\right> = \sum_{\alpha\in \Phi^+}\alpha(x)^2$. Since $\alpha(x)>0$ and $\alpha(ta_i)>0$ for all $\alpha\in\Phi^+$, $\sum_{\alpha\in \Phi^+}\alpha(x+ta_i)^2 > \sum_{\alpha\in\Phi^+}\alpha(x)^2$.
\end{proof}

\begin{lem}\label{grenzen}
Let $x\in \A_a$ be a vertex. Assume that for $y = x+\sum_{j=1}^n n_jc_ja_j\in x+\mathcal{C}$ one has $n_i=\alpha_i(y-x)\geq ht(\tilde{\Phi})sd(\gamma)+1$ for some $i\in \{1,\ldots,n\}$. Let $u\in U^-\cap U_x$. If $uy$ is fixed by $\gamma$, then $d(uy,x+c_ia_i)< d(uy,x)$. So if $uy$ is fixed by $\gamma$, then $uy$ is not above $x$.
\end{lem}
\begin{proof}
Let $\beta\in \tilde{\Phi}$. Since $uy$ is fixed by $\gamma$, $v(u_{-\beta}) \geq \beta(y)-ht(\beta)sd(\gamma)$ \cite[Proposition 4.2]{MS12}. Let $\alpha\in\Phi^+$. Write $$u_{-\alpha}=\prod_{\beta\in\rho^{-1}(\alpha)} u_{-\beta} \prod_{\beta \in \rho^{-1}(2\alpha)}u_{-\beta}.$$

Now $v(u_\alpha)=\min\{v(u_{-\beta}) : \beta\in\rho^{-1}(\alpha)\}\cup\{v(u_{-\beta})/2 : \beta\in\rho^{-1}(2\alpha)\}$.  The lower-bound for $v(u_{-\beta})$ and $\beta(y)=\rho(\beta)(y)$ give that
$$v(u_{-\alpha})\geq \alpha(y)-ht(\tilde{\Phi})sd(\gamma). $$
Now let $\alpha\in\Phi^+$ with a non-zero coefficient for $\alpha_i$ in the decomposition of $\alpha$ as linear combination of the simple roots in $\Delta$. So $\alpha = \sum_{j=1}^n d_j\alpha_j$ and $d_i\geq 1$. 
\begin{align*}
 v(u_{-\alpha}) &\geq \alpha(y)-ht(\tilde{\Phi})sd(\gamma)= \alpha(x)+\alpha(y-x)-ht(\tilde{\Phi})sd(\gamma)\\
 &\geq \alpha(x)+d_i\alpha_i(y-x)-ht(\tilde{\Phi})sd(\gamma)\\
 &\geq \alpha(x)+d_i(ht(\tilde{\Phi})sd(\gamma)+1)-ht(\tilde{\Phi})sd(\gamma)
 \geq \alpha(x)+d_i.
\end{align*}

For all $\alpha\in\Phi^+$ one has $v(u_{-\alpha}) \geq \alpha(x)$, since $u$ fixes $x$. Therefore with the previous inequality $v(u_{-\alpha})\geq \alpha(x+c_ia_i)$ for all $\alpha\in\Phi^+$. We conclude that $u$ fixes $x+c_ia_i$. Hence $$d(uy,x+c_ia_i)=d(y,u^{-1}(x+c_ia_i)) = d(y,x+c_ia_i) = d(y-c_ia_i,x).$$
Since $n_i\geq 1$, $y-c_ia_i\in x+\mathcal{C}$. So by Lemma \ref{gwva} $$d(uy,x+c_ia_i)=d(y-c_ia_i,x)<d(y,x)=d(uy,x).$$
Since $\alpha_k(c_ia_i) = \frac{\delta_{ki}c_i}{c_i} \in \N$ for all simple roots $\alpha_k$, the translation $y\mapsto y+c_ia_i$ is an automorphism of the apartment. So $x+c_ia_i$ is a vertex in $\A_a$.
\end{proof}

For $\alpha\in \Phi$ define $n_\alpha$ to be the smallest $r\in \R_{>0}$ such that $U_{\alpha,r} \not= U_{\alpha,r+}$. For $r\in \R$ define the $\alpha$-ceiling as: $\lceil r\rceil_\alpha := \min \{z\in n_\alpha\Z\mid z\geq r\}$.
\begin{lem}
Let $x\in \A_a$ and $y\in \overline{C}$. There is a system of positive roots $\Phi^{++}$ such that:\\
 $-\alpha(x) \geq -\alpha(y)$ and $\lceil-\alpha(x)\rceil_\alpha\geq \lceil f_C(\alpha)\rceil_\alpha$ for all $\alpha \in \Phi^{++}$\\
 $-\alpha(x) \leq -\alpha(y)$ and $\lceil-\alpha(x)\rceil_\alpha\leq \lceil f_C(\alpha)\rceil_\alpha$ for all $\alpha \in -\Phi^{++}=\Phi^{--}$
\end{lem}
\begin{proof}
First there is a construction of $\Phi^{++}\subset \Phi$, then it will be proven that it is a system of positive roots that satisfies the requirements.
For $\alpha \in \Phi^+$ the following rules decide whether $\alpha \in \Phi^{++}$ or $-\alpha \in \Phi^{++}$.\\

If $0<\alpha(y)\leq n_\alpha$:
\begin{align*}
 \alpha(x) < \alpha(y) &\Rightarrow +\alpha\in \Phi^{++}\\
 \alpha(x)\geq \alpha(y) & \Rightarrow -\alpha\in\Phi^{++}
\end{align*}

If $0=\alpha(y)$:
\begin{align*}
 \alpha(x) \leq \alpha(y) &\Rightarrow +\alpha\in \Phi^{++}\\
 \alpha(x) > \alpha(y) & \Rightarrow -\alpha\in\Phi^{++}
\end{align*}

\noindent By definition of $C$ one has $\lceil f_C(\alpha)\rceil_\alpha = n_\alpha$ if $\alpha \in \Phi^-$ and $0$ if $\alpha \in \Phi^+$.

First we check that the roots of $\Phi^{++}$ satisfy the requirements:\\
Certainly if $\alpha\in \Phi^{++}$, then $-\alpha(x)\geq -\alpha(y)$.\\
Let $\alpha\in \Phi^+$.\\
If $\alpha\in\Phi^{++}$, one has $\alpha(x)< n_\alpha$.\\  
If $\alpha(x)< n_\alpha$, then $\lceil -\alpha(x) \rceil_\alpha \geq 0 = \lceil f_C(\alpha)\rceil_\alpha$ and $\lceil \alpha(x)\rceil_\alpha \leq n_\alpha =\lceil f_C(-\alpha)\rceil_\alpha$.\\
If $-\alpha\in\Phi^{++}$, one has $\alpha(x)>0$.\\
If $\alpha(x)>0$, then $\lceil \alpha(x) \rceil_\alpha \geq n_\alpha = \lceil f_C(-\alpha)\rceil_\alpha$ and $\lceil -\alpha(x) \rceil_\alpha\leq 0=\lceil f_C(\alpha)\rceil_\alpha  $.\\
Thus for $\alpha\in\Phi^{++}$, one has $\lceil-\alpha(x)\rceil_\alpha\geq \lceil f_C(\alpha)\rceil_\alpha$ and for $\alpha\in\Phi^{--}$ one has $\lceil-\alpha(x)\rceil_\alpha\leq \lceil f_C(\alpha)\rceil_\alpha$.\\

By definition of $\Phi^{++}$ if $-\alpha(x)>-\alpha(y)$, then $\alpha\in \Phi^{++}$.\\

Clearly the half of the roots are in $\Phi^{++}$ and $\Phi^{++}\cap -\Phi^{++}=\emptyset$. Therefore it is now enough to show that if $\alpha,\beta\in \Phi^{++}$ and $\alpha+\beta \in \Phi$, then $\alpha+\beta\in\Phi^{++}$.\\
Let $\alpha,\beta\in \Phi^+$. Let $i,j\in \{-,+\}$. Assume that one has $i\alpha+j\beta\in\Phi$ and $i\alpha,j\beta\in \Phi^{++}$. Case by case it can be shown that $i\alpha+j\beta\in\Phi^{++}$.
\end{proof}
\begin{ste}\label{representanten}
 Let $y\in\A_a$.\\
 Define $\Pi := \{ \Psi \subset \Phi \mid \Psi \text{ is a system of positive roots of } \Phi \}$. Define for $\Psi\in\Pi$ the group $U^\Psi$ as the group generated by $U_\alpha$ for $\alpha\in \Psi$. Then
 $$B(G) = \bigcup_{\Phi^+\in \Pi} \{ ux : x\in \A_a, u \in U^{-\Phi^+}_y \mid \forall_{\alpha\in \Phi^+}\;\alpha(x)\geq \alpha(y) \}.$$
\end{ste}
\begin{proof}
(See \cite[\S 4.1]{MS12}) Let $x\in B(G)$, choose a retraction $\rho$ to $\A_a$ centered in $C$. Take $\Phi^{++}$ a set of positive roots such that $-\alpha(\rho(x))\leq -\alpha(y)$ and $\lceil f_{\rho(x)}(\alpha))\rceil_\alpha \leq \lceil f_C(\alpha)\rceil_\alpha$ for $\alpha \in -\Phi^{++}$. Let $D$ be a chamber in $\A_a$ whose closure contains $\rho(x)$ and for $\alpha\in\Phi^{--}$ one has $\lceil f_D(\alpha)\rceil_\alpha=\lceil\alpha(x)\rceil_\alpha$. Now $\lceil f_C(\alpha)\rceil_\alpha\geq \lceil f_{\rho(x)}(\alpha)\rceil_\alpha=\lceil f_D(\alpha)\rceil_\alpha$ for $\alpha \in \Phi^{--}$. Therefore $U^{--}_C \subset U^{--}_D$. Since $N_C=N_D$, one has $P_C\subset U_C^{++}P_D$. Because $P_C$ acts transitively on the sets of apartments containing $C$ there exists a $u\in U^{++}_C$ such that $x=u\rho(x)$.
\end{proof}
(Notice that with the same proof Theorem \ref{representanten} holds with $\A_a$ substituted by $\A_e$.)\\

Now we have all the ingredients to prove Theorem \ref{Lgrens}.\\

{\it Proof of Theorem \ref{Lgrens}\\}
Let $x\in \A_a$ be a vertex and let $z$ be a vertex above $x$ fixed by $\gamma$.\\
According to Theorem \ref{representanten} there is a positive root system $\Phi^+$ and $u\in U^-$ such that $z=uy$ with $y\in \A_a$ and $\alpha(y)\geq \alpha(x)$ for $\alpha\in\Phi^+$. Take $\Delta = \{\alpha_1,\ldots,\alpha_n\}$ to be the set of simple roots of $\Phi^+$. Define for each root $\alpha_i$ a vertex $a_i$ in $\A_a$ in the following way. Let $\mathcal{C} := \{ y\in\A_a : \alpha(y)>0 \text{ for all } \alpha\in\Phi^+\}$. Hence $y = x + \sum_{i=1}^n n_ic_ia_i$ with $n_i\in \R_{\geq0}$. Since $\gamma$ fixes $uy$ and $uy$ is above $x$, according to Lemma \ref{grenzen} $n_i< ht(\tilde{\Phi})sd(\gamma)+1$ for all $i\in \{1,\ldots,n\}$. Since $\dim \A_a =n$, there is a $c\in \R$ such that for all $\gamma \in T\cap P_\oo$ the number of vertices in $y\in\A_a\cap (x+\mathcal{C})$ with $\alpha_i(y-x)<  ht(\tilde{\Phi})sd(\gamma)+1$ is bounded by  $c(ht(\tilde{\Phi})sd(\gamma)+1)^n$.\\
By Theorem \ref{ugrens} 
$$\#\{uy : u\in U^-\cap P_x \mid \gamma uy=uy\} \leq |D(\gamma)|^{-\frac12}.$$ Therefore there is a $c\in \R$ such that for all $\gamma\in T\cap P_\oo$ and all vertices $x\in \A_a$ the number of vertices fixed by $\gamma$ and above $x$ is bounded by $c (ht(\tilde{\Phi})sd(\gamma)+1)^n|D(\gamma)|^{-\frac12}$.\hfill $\square$\\

Define the fundamental domain $F_a$ for the action of $S$ on $\A_a$ as follows:
$$F_a := \{x\in\A_a \mid \forall s\in S[d(x,\oo)\leq d(x,s\oo)] \}.$$
For $\gamma \in Z_G(S)$ and $w\in B_a$ let
$$ L_{a,\gamma} := \{ x\in Gw \mid x \text{ is above a vertex in } F_a\text{ and } \gamma x=x\}.$$
\begin{gev}\label{LLgrens}
 There is a $c\in \R$ such that for all semi-simple regular $\gamma \in Z_G(S)\cap P_w$:
 $$|L_{a,\gamma}|\leq c (ht(\tilde{\Phi})sd(\gamma)+1)^n|D(\gamma)|^{-\frac12} $$
\end{gev}
\begin{proof}
Let $N\in\N$ be the number of vertices in $F_a$ and $C$ be the $C$ of Theorem \ref{Lgrens}. Then $c:= NC$ will do.
\end{proof}

\subsection{An upper-bound for the Weyl integral}

\begin{ste}\label{utwi}
 Let $h\in G$ and $x\in B(h)$. Then there is a $C\in \R$ such that for all regular semi-simple $\gamma\in\; ^GZ_G(S)$
 \[\int_{Z_G(\gamma)\backslash G} 1_{hP_{[x,hx]}}(g^{-1}\gamma g)dg \leq C(ht(\tilde{\Phi})sd(\gamma)+1)^{n} |D(\gamma)|^{-\frac12}.\]
\end{ste}
\begin{proof}
By conjugating $h$ with a suitable element of $G$, $x$ can be moved to $\A_e$.
Both sides are invariant under conjugation with $G$. So without loss of generality we assume that $\gamma \in Z_G(S)$. Define $T := Z_G(\gamma)$. If the integral is non-zero, there is a $g\in G$ such that $g^{-1}\gamma g \in hP_{[x,hx]}$. Then $d(\gamma) = d(g^{-1}\gamma g)=d(h)$ by equation \ref{diseq}. Thus without loss of generality we assume that $d(\gamma)=d(h)$.\\

Since $\gamma\in Z_G(S)$ and $x\in \A_e(S)$, $x\in B(\gamma)$. Thus by Lemma \ref{intnaarly} 
$$\int_{T\backslash G} 1_{\gamma P_x}(g^{-1}\gamma g)dg \leq |L_\gamma|.$$
So it is enough to show that $|L_\gamma|\leq C(ht(\tilde{\Phi})sd(\gamma)+1)^{n} |D(\gamma)|^{-\frac12}$.\\

Let $M$ be a Levi subgroup, such that $Z_G(S)\subset M$. Define 
$$Z_G(S)_M := \{ \gamma \in Z_G(S) \mid d(\gamma)=d(h), \gamma \text{ is regular semi-simple and } M_\gamma =M\}.$$
We will give an upper-bound for $|L_\gamma|$ for all $\gamma \in Z_G(S)_M$.

\begin{lem}\label{GOnMO}
 Let $x\in B_e(G)$ and let $M$ a Levi subgroup. Then $Gx \cap B_e(M)$ consists of finitely many $M$-orbits.
\end{lem}
\begin{proof}
If the Lemma holds for $M$ it also holds for $gMg^{-1}$. Thus without loss of generality we assume that $S\subset M$. If $gx \in B_e(M)$, there is an $m\in M$ such that $mgx\in \A_e$. Thus every $M$-orbit may and will be represented by a point in $\A_e$. Let $F_a$ be the fundamental domain of $S$ in $\A_e$. Then every $M$-orbit has at least one point in $F_a$. Since $F_a$ is bounded and there is an $r\in \R$ such that $d(z,z')\geq r$ for distinct $z,z'\in Gx$, there are only finitely many points of $Gx$ in $F_a$. So the number of $M$-orbits in $G\cap B_e(M)$ is finite.
\end{proof}

Recall the canonical map $\phi_M : B_e(M)\rightarrow B_a(M)$. Define
$$L_{x,\gamma}(M) := \{ y\in \phi_M(Gx\cap B_e(M)) \mid y\text{ is above a vertex of } F_a \text{ and } \gamma y = y\}.$$
By Corollary \ref{LLgrens} and Lemma \ref{GOnMO} there is a $c\in \R$ such that for all $\gamma \in Z_G(S)_M$:
$$|L_{x,\gamma}(M)|\leq c (ht(\tilde{\Phi}_{M})sd_{M}(\gamma)+1)^n|D_{M}(\gamma)|^{-\frac12}.$$ 
Let $Y_M := B_e(Z(M))=\A_e(Z(M))$.\\
Then $\A_e(M) = \A_a(M)\oplus Y_M$.\\
Define $\pi_M : B_e(M)\rightarrow Y_M$ by $(g,x+ y) \mapsto (g,y)$, for $x\in \A_a(M)$ and $y\in Y_M$.
Define $D := \pi_M(F_a)$.

\begin{lem}\label{LnLm}
There is a $c_0$ only depending on $M$ such that $|L_\gamma| \leq c_0|L_{x,\gamma}(M)|$.
\end{lem}
\begin{proof}
For $z\in L_{x,\gamma}(M)$ define
$$F(z) := \phi_{M}^{-1}(z)\cap L_\gamma.$$
Let $z'\in \A_a(M)$, $a\in \A_a(M)$ and $u\in U_a$, such that $z= uz'$ and $z'$ is above $a$. 
Let $v \in \phi^{-1}_{M}(z)\cap L_\gamma$. Then there is a $y\in Y_M$ such that $v= u(z'+ y)$. Let $a+ y'\in F_a$ such that $u (z+ y)$ is above $a+ y'$. By Lemma \ref{y=y'}, then $y=y'$. Thus if $v\in \phi^{-1}_{M}(z) \cap L_\gamma$, then $u^{-1}v \in (z'+ D) \cap Gx$.\\
Because there exists an $r\in\R_{>0}$ such that $d(z,z')>r$ for all distinct $z,z'\in Gx$, there exists a $N\in\N$ such that $|(z'+ D) \cap Gx|\leq N$ for all $z'\in \A_a(M)$.\\
Thus $|F(z)|\leq N$ and the Lemma follows.
\end{proof}

{\it Continuation of the proof of Theorem \ref{utwi}.} By Lemma \ref{LnLm} and Corollary \ref{LLgrens} for all Levi subgroups $M$ containing $S$, there is a $C_M\in \R_{>0}$ such that for all $\gamma\in Z_G(S)$ with $M_\gamma =M$: \[|L_\gamma|\leq C_M(ht(\tilde{\Phi}_{M_\gamma})sd_{M_\gamma}(\gamma)+1)^n|D_{M_\gamma}(\gamma)|^{-\frac12}.\]
By Lemma \ref{D/Dm} and the fact that there are only finitely many Levi subgroups containing $S$ there is a $C\in \R$ such that for all $\gamma\in Z_G(S)$ with $d(\gamma)=d(h)$: 
\[|L_\gamma|\leq C(ht(\tilde{\Phi})sd(\gamma)+1)^n|D(\gamma)|^{-\frac12}.\qedhere\]
\end{proof}
\begin{pro}\label{Iybound}
 Let $f\in C_c^\infty(G)$ and let $\omega\subset G$ be a compact subset of $G$. Then there exists a $C\in \R$ such that for all $\gamma\in Z_G(S)\cap \omega$ 
 $$|\int_{Z_G(\gamma)\backslash G} f(g^{-1}\gamma g) dg| \leq C(ht(\tilde{\Phi})sd(\gamma)+1)^{\dim \A_a}|D(\gamma)|^{-\frac12}.$$
\end{pro}
\begin{proof}
Let $\Omega \subset G$ be a compact subset. Let $T := Z_G(\gamma)$. Then there are $g_1,\ldots, g_m\in G$ and $x_1,\ldots,x_m\in B_e$ such that $d_{g_i}(x_i)=d(g_i)$ and $\Omega \subset \bigcup_{i=1}^m g_iP_{[x_i,g_ix_i]}$. Therefore
$$ \int_{T\backslash G} 1_\Omega(g^{-1}\gamma g)dg \leq \sum_{i=1}^m \int_{T\backslash G} 1_{g_iP_{[x_i,g_ix_i]}}(g^{-1}\gamma g)dg.$$
So it is enough to give an estimate for $\int_{T\backslash G} 1_{g_iP_{[x_i,g_ix_i]}}(g^{-1}\gamma g)dg$.\\
Take $h_i\in G$ such that $x_i\in h_i\A_e(S)$. Now apply Theorem \ref{utwi} to $x\in \A_e(hSh^{-1})$ and $\gamma \in Z_G(S)\subset \;^GZ_G(hSh^{-1})$.\\ Since $C_c^\infty(G)$ is spanned as $\mathbb{C}$-vector space by the $1_\Omega$ with $\Omega$ a compact subset of $G$, the proposition follows. 
\end{proof}

\section{Local summability of the character on $^G{T}$ ($S\subset T$)}
In this section we combine the upper-bounds for the Weyl integration formula and for the character of the representation to show that the character is locally summable on $^GT$ for a maximal torus $T$ containing a maximal split torus $S$. It turns out that it is enough to show that $sd^k$ is locally summable on $T$. Inspired by Harish-Chandra \cite[Lemma 43]{HC70} we show that even $sd^k|D|^{-\epsilon}$ is locally summable on every maximal $\F$-torus $T$ of $G$ for some $\epsilon>0$ depending on $T$.  
\subsection{Local summability of $sd^k|D|^{-\epsilon}$ on $T$}
In the first part of this subsection $T$ is an arbitrary $\F$-torus (not necessarily contained in $G$).\\
Integrating a function in a small neighborhood of the identity in a $\F$-split torus can be translated to integrating a function in a small neighborhood of $0$ in a $\F$-vector space. Just apply the map $e: \co \rightarrow \co^\times$, $e(a) := 1+\pi a$. If $\chi\in X^*(T)$, then integrating the function $|\chi(t)-1|^{-\epsilon}$ in a small neighborhood of $id$, becomes integrating  $|(1+\pi x)^n-1|^{-\epsilon}$ over a small neighborhood of $0$. To study the integral $|(1+\pi x)^n-1|^{-\epsilon}$ over $\co$, we want to have an estimate for the measure of \[\co_r := \{ x\in \co \mid v((1+\pi x)^n-1)\geq r\}\] in $\co$. For this we study first \[\co_r(f) := \{ x\in \co \mid v(f(x))\geq r\}\] for a polynomial $f\in \co[x]$, with $f\not=0$. In the case that $T$ is not an $\F$-split torus there is in general no polynomial bijection between $\co^m$ and a neighborhood of the identity. However, we are able to construct a surjective map from $\co_\E^n$ to $\cst$ for some Galois extension $\E$ and compact subgroup $\cst$ of $T$, using a generalised norm map $N_{\E/\F} : T(\E)\rightarrow T(\F)$. This gives rise to the study of the measure of \[\co^{n}_r(f) := \{ x\in \co^n \mid v(f(x))\geq r\}\] in $\co^n$ for a polynomial $f\in\co_\E[x_1,\ldots,x_n]$, with $f\not=0$.\\

For $f\in \F[x_1,\ldots, x_n]$, write $f = \sum_{a\in\N^n}c(a)\prod_{i=1}^n x_i^{a_i}$.\\
Define $m_i := \max \left\{ l \in \N \mid \exists a\in\N^n[a_i = l \text{ and } c(a)\not=0]\right\}$.\\
Define $m_f := \max_{i} m_i$.\\
Thus $m_f$ is the highest number that occurs as a power of any $x_i$ in the expression of $f$.\\
Thus for $f(x_1,x_2) := x_1x_2^3+x_1x_2+2$ we have $m_1=1$, $m_2=3$ and $m_f=3$.
\begin{lem}\label{epimv}
 Let $\E/\F$ be a finite field extension.\\
 Let $f \in \co_\E[x_1,\ldots,x_n]$ and $f\not=0$. There exists a $C\in\R_{>0}$ such that for all $r\in \Q$ and $N\in\N$ with $N\geq r$:
 $$\frac{1}{q^{nN}}|\{ x\in (\co_\F/\pi^N\co_\F)^n\mid v(f(x))\geq r\}| \leq CN^{n-1}q^{-\frac{r}{m_f}}.$$
\end{lem}

\begin{proof}
Since $f\in\co_\E[x_1,\ldots,x_n]$, to ask for $x\in (\co/\pi^N\co)^n$, whether $v(f(x))\geq r$ makes sense if $N\geq r$.\\
We prove this lemma with induction on $n$.\\
Assume that $n=1$, so $f(x) := \sum_{i=1}^m a_ix^i$, with $a_m\not=0$. Take $\alpha_1,\ldots,\alpha_n$ in an algebraic closure of $\E$ such that $f(x) = a_m\prod_{i=1}^m (x-\alpha_i)$.\\
Assume that $v(f(x))\geq r$. Then for some $i$, $v(a_m)+mv(x-\alpha_i)\geq r$. Thus $v(x-\alpha_i)\geq \frac{r-v(a_m)}{m}$.\\
So the number of $x\in \co_\F/\pi^N\co_\F$ such that $v(f(x))\geq r$ is bounded by $mq^{N-\frac{r-v(a_m)}{m}}$. Hence

$$\frac{1}{q^{N}}|\{ x\in \co_F/\pi^N\co_F\mid v(f(x))\geq r\}| \leq mq^{-\frac{r-v(a_m)}{m}}.$$

Assuming that we know the Lemma for $n$, we will prove the Lemma for $n+1$.\\
Let $m := m_f$. Without loss of generality assume that $m=m_{n+1}$. Take $g_0,\ldots,g_m\in \co_\E[x_1,\ldots,x_n]$ such that $f = \sum_{i=0}^m g_ix_{n+1}^i$. Then $g_m\not=0$ and $m\geq m_{g_m}$. Now we apply the induction hypothesis on $g_m$. Take a $C\in\R_{>0}$ such that for all $r\in\Q$ and $N\in \N$ with $N\geq r$, 
$$\frac{1}{q^{nN}}|\{ x\in (\co_\F/\pi^N\co_\F)^n\mid v(g_m(x))\geq r\}| \leq CN^{n-1}q^{-\frac{r}{m_{g_m}}}.$$

Define the following sets
\begin{align*}
V_r &:= \{ x\in (\co/\pi^N\co)^n \mid v(g_m(x))= r\},\\
O_{r,s} &:= \{x\in (\co/\pi^N\co)^{n+1}\mid v(g_m(x_1,\ldots,x_n))= s \text{ and }  v(f(x))\geq r\}.
\end{align*}
Define, for $x_1,\ldots, x_n \in \co/\pi^N\co$ and $r\in\Q$, the set:
$$ U_{x_1,\ldots,x_n,r} := \{ x\in \co/\pi^N\co \mid v(f(x_1,\ldots,x_n,x))\geq r\}.$$
So
$$O_{r,s} = \{ x \in (\co/\pi^N\co)^{n+1} \mid (x_1,\ldots,x_n)\in V_{s}\text{ and } x_{n+1}\in U_{x_1,\ldots,x_n,r}\}.$$
By the proof of the lemma in the case $n=1$ we have 
$$|U_{x_1,\ldots,x_n,r}| \leq mq^{N-\frac{r-v(g_m(x_1,\ldots,x_n))}{m}}$$
whenever $v(g_m(x_1,\ldots,x_n))<N$.

Let $x_1,\ldots,x_n\in\co/\pi^N\co$, such that $v(g_m(x_1,\ldots,x_n))=s< N$. Then
$$\frac{1}{q^N}|U_{x_1,\ldots,x_n,r}|\leq mq^{-\frac{r-s}{m}}.$$
By the induction hypothesis on $g_m$ we have 
$$\frac{1}{q^{nN}}|V_s| \leq CN^{n-1}q^{-\frac{s}{m_{g_m}}}.$$
Thus
\begin{align*}
 \frac{1}{q^{(n+1)N}} |O_{r,s}| &= \frac{1}{q^{(n+1)N}} \sum_{x\in V_s} |U_{x,r}|\leq \frac{1}{q^{nN}} \sum_{x\in V_s} mq^{-\frac{r-s}{m}}\\
 &= \frac{1}{q^{nN}}|V_s| mq^{-\frac{r-s}{m}}\leq CN^{n-1}q^{-\frac{s}{m_{g_m}}}mq^{-\frac{r-s}{m}}\\
 &\leq mCN^{n-1}q^{-\frac{r}{m}}.
\end{align*}
Let $e$ be the ramification index of $\E/\F$. So
\begin{align*}
&\frac{1}{q^{(n+1)N}}|\{ x\in (\co_\F/\pi^N\co_\F)^{n+1}\mid v(f(x))\geq r\}|\\
&\leq \frac{1}{q^{(n+1)N}}|\{x\in (\co_\F/\pi^N\co_\F)^{n+1}\mid v(g_m(x_1,\ldots,x_n))\geq N\}|+\sum_{i=0}^{eN-1}\frac{1}{q^{(n+1)N}} |O_{r,\frac{i}{e}}|\\
&\leq \frac{1}{q^{nN}}|\{ x\in (\co_\F/\pi^N\co_\F)^{n}\mid v(g_m(x))\geq N \}|+\sum_{i=0}^{eN-1}\frac{1}{q^{(n+1)N}} |O_{r,\frac{i}{e}}|\\
&\leq CN^{n-1}q^{-\frac{N}{m_{g_m}}}+\sum_{i=0}^{eN-1}mCN^{n-1}q^{-\frac{r}{m}} \\
&\leq CN^{n-1}q^{-\frac{r}{m}}+eNmCN^{n-1}q^{-\frac{r}{m}}\\
&\leq 2emC N^nq^{-\frac{r}{m}}. \qedhere
\end{align*}
\end{proof}
Let $\E/\F$ be a finite Galois extension such that $T$ is $\E$-split. Define the function $N_{\E/\F}: T(\E)\rightarrow T(\F)$ as follows:
$$ N_{\E/\F}(t) := \prod_{\sigma\in Gal(\E/\F)} \sigma(t).$$
Since $T$ is Abelian, $N_{\E/\F}(t)$ is invariant under the Galois action. Hence the image of $N_{\E/\F}$ lies in $T(\F)$. The group $\text{Gal}(\E:\F)$ acts on $X^*(T)$ by
\[ (\sigma\cdot \chi)(t) := \sigma(\chi(\sigma^{-1}(t))).\]

Let $m = \dim T$ and $n= [\E:\F]$.\\

Let $\chi_1,\ldots,\chi_m$ be a basis for $X^*(T)$ and $X_1,\ldots,X_m$ the dual basis for $X_*(T)$. Parametrize $T(\E)$ by $(\E^\times)^n\rightarrow T(\E)$:
$$ a\mapsto \prod_{i=1}^m X_i(a).$$
Define $K := \{ \prod_{i=1}^n X_i(a_i) : a_i\in 1+\pi\co_\E\}$.\\

Take $\alpha\in\E$ such that $\co_\E = \co_\F[\alpha]$.\\
Define $\alpha_i := \alpha^{i-1}$ for $i=1,\ldots,n$. So
\[
\sum_{i=1}^n a_i\alpha_i \in \pi^k\co_\E \Leftrightarrow \forall i[ a_i\in\pi^k\co_\F].
\]
Write $\E$ as $\F$-vector space with basis $1,\alpha_2,\ldots,\alpha_n$. For $a\in \E^m$, we define, for $1\leq i\leq m$ and $1\leq j\leq n$, the elements $a_{ij}\in\F$ to be the coordinates of $a_i$ with respect to this basis. Thus
$$a_i = \sum_{j=1}^n a_{ij}\alpha_i.$$
Define $\mathsf{p} : (\F^{n}-0)^m\rightarrow T(\E)$ by:
$$\mathsf{p}(a) := \prod_{i=1}^m X_i(\sum_{j=1}^n a_{ij}\alpha_j).$$
\begin{lem}\label{polyuit}
 Let $\chi\in X^*(T)$. There exist $f,g\in \co_\E[x_{11}, \ldots, x_{mn}]$, such that 
 $$\chi\circ N_{\E/\F}\circ \mathsf{p} (a) = \frac{f(a)}{g(a)}.$$
 Moreover if $\mathsf{p}(a)\in K$, then $f(a),g(a)\in \co_\E^\times$.
\end{lem}
\begin{proof}
Let $a\in\E^\times$, then
\[\chi\circ N_{\E/\F}\circ X_i(a) =\prod_{\sigma\in \text{Gal}(\E/\F)} \chi(\sigma(X_i(a)))= \prod_{\sigma\in \text{Gal}(\E/\F)}\sigma(a)^{z_\sigma},\]
where $z_\sigma=\left<\sigma^{-1}\cdot \chi,X_i\right>$. An automorphism $\sigma\in\text{Gal}(\E/\F)$ is, with $\E$ viewed as $\F$-vectorspace with basis $1,\alpha_2,\ldots,\alpha_n$, a polynomial map:
$$g_\sigma(x_1,\ldots,x_n):= \sum_{i=1}^{n} \sigma(\alpha_i)x_i.$$
Then for $a = \sum_{i=1}^n a_i\alpha_i$ with $a_i\in \F$ we have
$$g_\sigma(a_1,\ldots,a_n)= \sigma(a).$$
Since $\alpha\in\co_\E$ also $\sigma(\alpha^i)\in \co_\E$ for $i\geq 0$. Therefore $g_\sigma\in\co_\E[x_1,\ldots,x_n]$.\\
Thus $$\chi\circ N_{\E/\F}\circ X_i(\sum_{j=1}^na_j\alpha_j)=\prod_{\sigma\in\text{Gal}(\E/\F)}g_\sigma(a_1,\ldots,a_n)^{z_\sigma}=\frac{f_i(a_1,\ldots,a_n)}{g_i(a_1,\ldots,a_n)},$$
where $f_i,g_i\in \co_\E[x_1,\ldots,x_n]$. The first part of the lemma follows.\\
If $a \in 1+\pi\co_\E$, then $\sigma(a)\in 1+\pi\co_\E$ for all $\sigma\in\text{Gal}(\E/\F)$. Thus if $\mathsf{p}(a)\in K$, then $f(a),g(a)\in\co_\E^\times$.
\end{proof}
Fix $\chi\in X^*(T)$.\\
For $r\in \frac{1}{e}\N$ define
\[K_r := \{ k\in K : v(\chi\circ N_{\E/\F}(k)-1)\geq r\}.\]
Then $K_r$ is a compact open subgroup of $K$.
\begin{lem}\label{afkkr}
 There exist $c_1, c_2\in \R_{> 0}$ such that 
 $$\frac{1}{[K:K_r]} \leq c_1\lceil r\rceil^{nm-1}q^{-\frac{r}{c_2}},$$
 for all $r\in \frac{1}{e}\N$.
\end{lem}
\begin{proof}
Take $f,g\in \co_\E[x_{11},\ldots, x_{mn}]$ as in Lemma \ref{polyuit}.\\

Since the elements of $T(\F)$ are invariant under the Galois action:
$$ \chi \circ N_{\E/\F} |_{T(\F)} = n\chi|_{T(\F)}.$$
Since $T(\F)$ is Zariski dense in $\mathcal{T}$, there is a $t\in T(\F)$ such that $n\chi(t)\not=1$. So there is a $t\in T(\F)$ such that $\chi \circ N_{\E/\F}(t) \not=1$. Thus $\frac{f(a)}{g(a)}\not=1$.\\

Let $\mathbf{e} :\co_\E\rightarrow 1+\pi\co_\E$ by $\mathbf{e}(a) := 1+\pi a$.\\
Let $\mathsf{p}': (\co_\F^n)^m\rightarrow K$ defined by 
$$\mathsf{p}' (a) := \prod_{i=1}^m X_i(\mathbf{e}(\sum_{j=1}^n a_{ij}\alpha_i)).$$
Then $\mathsf{p}'$ is a bijection.
Now
$$\chi\circ N_{\E/\F}\circ \mathsf{p}'(a)= \frac{\psi(f)(a)}{\psi(g)(a)},$$
where $\psi: \E[x_{11},\ldots,x_{mn}]\rightarrow \E[x_{11},\ldots,x_{mn}]$ is the automorphism defined by $$\psi(x_{ij}) := \left\{ \begin{array}{cl} 1+\pi x_{ij} & \text{if } j=1 \\ \pi x_{ij} & \text{otherwise.}\end{array}\right.$$

The bijection $\mathsf{p}'$ gives a set corresponding to $K_r$ in $(\co_\F^n)^m$:
$$(\co_\F^n)^m_r := \mathsf{p}'^{-1}(K_r) = \{ a\in (\co_F^n)^m \mid v\left(\frac{\psi(f)(a)}{\psi(g)(a)}-1\right)\geq r\}.$$

Since for all $x\in(\co_\F^n)^m$, $\psi(g)(x)\in\co^\times$ we have $$v\left(\frac{\psi(f)(x)}{\psi(g)(x)}-1\right)=v(\psi(f)(x)-\psi(g)(x)).$$
Define $h(x) := \psi(f)(x)-\psi(g)(x)$, then $h\in \co_\E[x_{11},\ldots, x_{mn}]$ and
$$(\co_\F^n)^m_r=\{ a\in ((\co_\F)^n)^m \mid v(h(a))\geq r\}.$$
Define $K^{(N)} := \{ \prod_{i=1}^n X_i(a_i) : a_i\in 1+\pi^N\co_\E\}$.\\
Now $\mathsf{p'}(a) K^{(N)} = \mathsf{p'}(a') K^{(N)}$ if and only if $a_{ij}\equiv a'_{ij} \mod \pi^{N-1}\co_\F$. Let $N\geq r$, then $K^{(N)}<K_r$. Thus

$$\frac{1}{[K:K_r]} = \frac{1}{q^{nm(N-1)}}|\{ x\in (\co_\F/\pi^{N-1}\co_\F)^{mn}\mid v(h(x))\geq r\}|.$$ 
By Lemma \ref{epimv} there exists a $C$ such that for all $r$ and $N$ with $N\geq r$,
$$\frac{1}{q^{nmN}}|\{ x\in (\co_\F/\pi^N\co_\F)^{mn}\mid v(h(x))\geq r\}|\leq CN^{nm-1}q^{-\frac{r}{m_h}}.$$
Take $N = \lceil r \rceil +1$. Thus $\frac{1}{[K:K_r]}\leq C(\lceil r\rceil+1)^{nm-1}q^{-\frac{r}{m_h}}$.
\end{proof}

Define $\cst:= N_{\E/\F}(K)$. Since $K$ is compact, $\cst$ is a closed subgroup of $T(\F)$. We have $\cst<K$. Let $\cst_r := \{ s\in \cst \mid v(\chi(s)-1)\geq r\}$.
Define $T(\F)_r := \{ t\in T(\F) : v(\chi(t)-1)\geq r\}$.
\begin{lem}\label{afttr}
 $[K : K_r]=[\cst:\cst_r]\leq [T(\F)\cap K:T(\F)_r\cap K].$
\end{lem}
\begin{proof}
Since $N_{\E/\F}: K\rightarrow \cst$ is surjective,
\[ [\cst: \cst_r] = [K : K_r \ker N_{\E/\F}].\]
If $k \in \ker N_{\E/\F}$, then $\chi\circ N_{\E/\F}(k)=\chi(1)=1$. Thus $\ker N_{\E/\F}< K_r$. So $K_r \ker N_{\E/\F} = K_r$. Thus $[K:K_r]=[\cst:\cst_r]$.\\
Since $\cst<T(\F)\cap K$ and $\cst_r = T(\F)_r\cap K \cap \cst$, 
\[[\cst:\cst_r]\leq [T(\F)\cap K:T(\F)_r\cap K].\qedhere\]
\end{proof}

\begin{pro}\label{Xls}
 There exists an $\epsilon>0$ such that $|\chi(t)-1|^{-\epsilon}$ is locally summable on $T(\F)$.
\end{pro}

\begin{proof}
Let $t_0\in T$. If $\chi(t_0)\not=1$, then $|\chi(t)-1|^{-\epsilon}$ is constant on a neighborhood of $t_0$. Thus in particular $|\chi(t)-1|^{-\epsilon}$ is locally summable around $t_0$.\\
Assume that $\chi(t_0)=1$ and $\int_K|\chi(t)-1|^{-\epsilon}dt <\infty$ for some open compact subgroup $K<T$. Since $\chi(t_0t)=\chi(t)$ for $t\in K$,
$$\int_{t_oK}|\chi(t)-1|^{-\epsilon}dt=\int_K|\chi(t)-1|^{-\epsilon}dt<\infty. $$
So then $|\chi(t)-1|^{-\epsilon}$ is locally summable around $t_0$. Thus it is enough to show that for some open compact subgroup $K$
 $$\int_{K\cap T(\F)} |\chi(t)-1|^{-\epsilon}dt<\infty.$$
Take $K$ as before. Change $\mu$ such that $\mu(T(\F)\cap K)=1$. Take $c_1,c_2\in\R_{>0}$ as in Lemma \ref{afkkr}. Then
\begin{align*}
 \int_{K\cap T(\F)} |\chi(t)-1|^{-\epsilon}dt&\leq \sum_{s=0}^\infty q^{\epsilon\frac{s}{e}}\mu(T(\F)_{\frac{s}{e}})\leq \sum_{s=0}^\infty q^{\epsilon\frac{s}{e}}\frac{1}{[K : K_{\frac{s}{e}}]}\\
 &\leq \sum_{s=0}^\infty q^{\epsilon\frac{s}{e}}C\left\lceil\frac{s}{e}\right\rceil^{c_1} q^{-\frac{1}{c_2}\frac{s}{e}} =\sum_{s=0}^\infty C\left\lceil\frac{s}{e}\right\rceil^{c_1} q^{(\epsilon-\frac{1}{c_2})\frac{s}{e}},
\end{align*}
where the second inequality is due to Lemma \ref{afttr}. The last sum converges if $\epsilon<\frac{1}{c_2}$.
\end{proof}

From now on $T$ is a maximal $\F$-torus in $G$. Let $R(G,T)$ be the roots of $T$ and $G$. Define \[M := \max_{\alpha\in R(G,T)}\max_{i=1}^m \left<\alpha,X_i\right>.\]
\begin{gev}\label{Xls2}
 Let $\alpha\in R(G,T)$. The function $|\alpha(t)-1|^{-\epsilon}$ is locally summable on $T(\F)$ for $\epsilon < \frac{1}{M[\E;\F]}$.
\end{gev}
\begin{proof}
By the proof of Proposition \ref{Xls}, if $\epsilon<\frac{1}{c_2}$ for the $c_2$ of Lemma \ref{afkkr}, the function $|\alpha(t)-1|^{-\epsilon}$ is locally summable. The $c_2$ of Lemma \ref{afkkr} is equal to $m_h$, where $h = \psi(f)-\psi(g)$ for the $g$ and $f$ of Lemma \ref{afkkr}. Therefore $m_h\leq \max(m_f,m_g)$. The proof of Lemma \ref{polyuit} shows that 
\begin{align*}
f(x_{11},\cdots,x_{mn})&=\prod_{i=1}^m f_i(x_{i1},\cdots,x_{in}),\\
g(x_{11},\cdots,x_{mn})&=\prod_{i=1}^m g_i(x_{i1},\cdots,x_{in}).
\end{align*}
Thus $m_f = \max_{i=1}^m m_{f_i}$ and $m_g=\max_{i=1}^m m_{g_i}$.\\
The $f_i$ and $g_i$ are such that
\[\alpha\circ N_{\E/\F}\circ X_i(\sum_{j=1}^na_j\alpha_j)=\prod_{\sigma\in\text{Gal}(\E/\F)}g_\sigma(a_1,\ldots,a_n)^{z_{i,\sigma}}=\frac{f_i(a_1,\ldots,a_n)}{g_i(a_1,\ldots,a_n)},\]
with $z_{i,\sigma} = \left<\sigma^{-1}\cdot\alpha,X_i\right>$. Therefore
\[ \max(m_{f_i},m_{g_i}) \leq \sum_{\sigma \in \text{Gal}(\E:\F)} |z_{i,\sigma}| \leq [\E:\F] M,\]
since $\sigma^{-1}\cdot \alpha\in R(G,T)$ for all $\sigma \in \text{Gal}(\E:\F)$. Thus \[m_h\leq \max(m_f,m_g) = \max(\max_{i=1}^m m_{f_i},\max_{i=1}^m m_{g_i}) \leq [\E:\F]M.\qedhere\]
\end{proof}

\begin{lem}\label{pall}
 Let $X$ be a space with measure $\mu$ and let $f:X\rightarrow \R_{\geq0}$ and $g:X\rightarrow \R_{\geq 0}$. Assume that $f^{-\epsilon}$ and $g^{-\epsilon}$ are locally summable if $0<\epsilon<\epsilon_o$. Then $(fg)^{-\epsilon}$ is locally summable if $0<\epsilon<\frac{\epsilon_o}{2}$.
\end{lem}
\begin{proof}
If $f^{-\epsilon}$ is locally summable for all $\epsilon<\epsilon_o$, then $(f^2)^{-\epsilon}$ is locally summable for all $\epsilon<\frac{\epsilon_o}{2}$. Thus $f^{-\epsilon}$ and $g^{-\epsilon}$ are locally square integrable for all $\epsilon<\frac{\epsilon_o}{2}$. Then $(fg)^{-\epsilon}$ is locally summable for all $\epsilon<\frac{\epsilon_o}{2}$.
\end{proof}

\begin{ste}\label{onsls}
 If $\epsilon<\frac{1}{2^{|R(G,T)|-1}M[\E:\F]}$, then $|D(t)|^{-\epsilon}$ is locally summable on $T$.\\
 Moreover if $\epsilon<\frac{1}{2^{|R(G,T)|}M[\E:\F]}$, then, for all $n\in \Z_{\geq 0}$, the function $sd(\gamma)^n|D(t)|^{-\epsilon}$ is locally summable on $T$.
\end{ste}
\begin{proof}
That $|D(t)|^{-\epsilon}$ is locally summable on $T$ for $0<\epsilon<\frac{1}{2^{|R(G,T)|-1}M[\E:\F]}$ follows from the Corollary \ref{Xls2} and Lemma \ref{pall}.

First we show that $sd_\alpha^n$ is locally summable for all $n\in\Z_{\geq 0}$.\\
Let $t_o\in T$. If $\alpha(t)\not= 1$, then $sd_\alpha$ is locally constant on $t_o$ and hence $sd_\alpha^n$ is locally summable around $t_o$.\\
If $\alpha(t_o)=1$, then let $U := \alpha^{-1}(\co)$ be a neighborhood of $t_o$. So it is enough to show that $sd_\alpha^n$ is locally summable in $U$.\\

By Proposition \ref{Xls} there is an $\epsilon>0$ such that $|\alpha(t)-1|^{-\epsilon}$ is locally summable on $T$. Since $|\alpha(t)-1|^{-1} = q^{sd_\alpha(t)}$ if $v(\alpha(t)-1)\geq 0$, there is a $N\in \N$ such that $sd_\alpha(t)^n\leq N|\alpha(t)-1|^{-\epsilon}$ for all $t\in U$. Thus $sd_\alpha(t)^n$ is locally summable on $U$, since $N|\alpha(t)-1|^{-\epsilon}$ is.\\

If $0<\epsilon<\frac{1}{2^{|R(G,T)|}M[\E:\F]}$, then $|D(t)|^{-2\epsilon}$ is locally summable by the first statement of this Theorem. Since $sd_\alpha(t)^{n}$ is locally summable for all $n\in \Z_{\geq 0}$, also $sd(t)^{2n}$ is locally summable for all $n\in\Z_{\geq 0}$. Thus $sd(t)^n|D(t)|^{-\epsilon}$ is locally summable for $0<\epsilon<\frac{1}{2^{|R(G,T)|}M[\E:\F]}$, because $sd(t)^{2n}$ and $|D(t)|^{-2\epsilon}$ are locally summable.
\end{proof}
In the case that $\text{char } \F=0$ and $\epsilon$ is small, Theorem \ref{onsls} has been proven by Harish-Chandra in \cite[Lemma 43]{HC70}.

\subsection{Local summability of the character}
\begin{lem}\label{iccwtic}
 Let $\omega\subset G$ be compact and $T$ a maximal torus. Then $^G\omega\cap T$ is  contained in a compact subset of $T$, i.e. is bounded.
\end{lem}
\begin{proof}
Assume that $T$ is $\F$-split.\\ 
Let $d:G\rightarrow \R$ be the displacement function of $B_e$.\\
CLAIM: For each $r\in \R$ the set $\{ t\in T \mid d(t)\leq r\}$ is bounded.\\
By the proof of \cite[Proposition 7.4.25]{BT72} there is a retraction $\rho:B_e\rightarrow \A_e$ defined by $y=u\rho(y)$ for some $u\in U^+$. Now $\rho$ is $T$-equivariant: \\ $tux=tut^{-1}tx$, so $\rho(tux)=tx=t\rho(ux)$. Thus \[d(x,tx)=d(\rho(ux),\rho(tux))\leq d(ux,tux).\] Thus $d(t) = d(tx,x)$ for $x\in \A_e(T)$. Therefore $d(t) = d(v(t),\oo)$. Since there are only finitely many points $x\in T\oo$ with $d(x,\oo)\leq r$, the set $\{ t\in T \mid d(t)\leq r\}$ is bounded.\\

The function $g\mapsto d(g)$ is a continuous class function, see \cite{Mo00} and \cite[Lemma 3.4.7]{DB02b}. Thus the image of $^G\omega$ is compact in $\R$. So $^G\omega\cap T$ is bounded.\\

Now we go to the general case:\\
Let $\E$ be a field extension of $\F$ such that $\ct$ is $\E$-split. Since $^{\cg(\E)}\omega\cap \ct(\E)$ is bounded and $^G\omega\cap T\subset \;^{\cg(\E)}\omega\cap \ct(\E)$, also $^G\omega\cap T$ is bounded.
\end{proof}

If $\omega$ is compact modulo $Z(G)$, then $^G\omega\cap T$ is also compact modulo $Z(G)$. This could be proven in the same way as Lemma \ref{iccwtic}; by the displacement function on the reduced building. There is in this case a more elementary proof using $g\mapsto \det(ad(g)-x)$, see \cite[Lemma 39]{HC70}.

\begin{pro}\label{ontglt}
 Let $T$ be a maximal torus of $G$ containing a maximal split torus $S$. The function $\gamma \mapsto (ht(\tilde{\Phi})sd(\gamma)+1)^m |D(\gamma)|^{-\frac{1}{2}-\epsilon}$ is locally summable on $^GT$ for small $\epsilon\geq 0$.
\end{pro}
\begin{proof}(See \cite[VII,\S1]{HC70}) Let $\omega\subset G$ be compact open. By the Weyl integration formula and Proposition \ref{Iybound}: 
\begin{align*}
 \int_{^GT} &1_{\omega}(g)(ht(\tilde{\Phi})sd(g)+1)^m |D(g)|^{-\frac{1}{2}-\epsilon}dg\\
 &= |W|^{-1}\int_T |D(t)|\int_{T\backslash G}1_\omega(g^{-1}tg)(ht(\tilde{\Phi})sd(g^{-1}tg)+1)^m |D(g^{-1}tg)|^{-\frac{1}{2}-\epsilon}dgdt\\
 &= |W|^{-1}\int_T |D(t)|\int_{T\backslash G}1_\omega(g^{-1}tg)(ht(\tilde{\Phi})sd(t)+1)^m |D(t)|^{-\frac{1}{2}-\epsilon}dgdt\\
 &\leq C \int_T 1_{\Omega}(t) |D(t)|^{\frac{1}{2}} (ht(\tilde{\Phi})sd(t)+1)^{n+m} |D(t)|^{-\frac{1}{2}-\epsilon}dt,
\end{align*}
where $\Omega\subset T$ is compact and $^G(\omega)\cap T \subset \Omega$ (see Lemma \ref{iccwtic}). The right-hand-side is finite by Theorem \ref{onsls}.
\end{proof}

\begin{ste}\label{theresult}
 Let $(\rho,V)$ be a $G$-representation of finite length with character $\theta$ and $f \in C^\infty_c(G)$, then for every torus $T$ containing $S$:
 $$\int_{^GT} f(g)\theta(g) dg <\infty.$$
\end{ste}
\begin{proof} This follows from Propositions \ref{afschattingalgemeen} and \ref{ontglt}.\end{proof}

The following Corollary has already been proven by Van Dijk in \cite[Theorem 3]{Di72} by other means.
\begin{gev}
 Assume that $G$ is quasi-split. Let $\chi$ be a representation of $T:= Z_G(S)$ of finite length. Then the character of $ind_B^G(\chi)$ is locally summable.
\end{gev}
\begin{proof}
Since $G$ is quasi-split, $T$ is a maximal torus. Let $\gamma$ be a regular semi-simple element not in $^GT$. Let $K<G$ be a compact open subgroup such that $\gamma K\subset\, ^G{Z_G(\gamma)}$ (see \cite[Lemma 6.5]{MS12} for a specific $K$). Let $K_o := K\cap \gamma K\gamma^{-1}$, then $K_o\gamma K_o\subset \gamma K\subset ^G{Z_G(\gamma)}$. Thus $^GT\cap K_o\gamma K_o=\emptyset$. So for every open compact subgroup $K'\subset K_o$ we have 
$TgK'\cap TgK'\gamma K'=\emptyset$ for all $g\in G$. Since the character of $ind_B^G \chi$ is supported on $^GT$,
$$tr(ind_B^G(\chi) (e_{K'}*\gamma * e_{K'}))=0.$$
Hence the character of the induced representation is zero on the regular semi-simple elements outside $^GT$. Now apply Theorem \ref{theresult}.
\end{proof}
\section{Future work}
This article is based on the study of fixed points in the reduced and extended building of compact regular semi-simple element in the centralizer of a maximal split torus. The understanding of the distribution of these fixed points gives the estimates for the character of an admissible smooth representation and the Weyl integration formula. In the last chapter we saw that both upper-bounds are small enough to prove that the character of a finite length representation is locally summable on $^G{T}$, for $T$ containing a maximal split torus. 

A study of fixed points for general regular semi-simple elements should lead to similar estimates. We hope that these upper-bounds can be chosen small enough to prove that for every maximal torus $T$ the character is locally summable on $^GT$. In the case that there are finitely many conjugacy classes of tori the locally summability of the character follows from the locally summability on $^GT$. However in positive characteristic there could be infinitely many conjugacy classes of tori. In that case the estimates should be synchronized in some way.

In the last section we introduced a generalization of the norm map $N_{\E/\F} : T(\E)\rightarrow T(\F)$. It would be interesting to see whether this map has analogous properties as the regular norm map. In particular whether the norm map is open and whether $[T(\F):N_{\E/\F}(T(\E))]<\infty$. 

\section{Appendix}
In this appendix we give a proof of the Weyl integration formula:
\begin{ste}[The Weyl Integration Formula]\label{WIF}
 Let $G$ be a $p$-adic reductive group, $T$ a maximal torus of $G$ and $W=N_G(T)/T$ its Weyl group. Assume that the measures on $G$, $T$ and $G/T$ are such that for all $f\in C_C^\infty(G)$:
 \[ \int_Gf(g)dg = \int_{G/T} \int_T f(gt)dtdg^*.\]
 Let $f\in C_c^\infty(G)$, then
 \[ \int_{^GT} f(g) dg = \frac{1}{|W|}\int_T \int_{G/T} f(gtg^{-1})dg^*dt.\]
\end{ste}

The Weyl integration formula is a well-known result in the theory of reductive groups. However the author could not find a "spelled out" proof of the formula for the $p$-adic case in the literature. Harish-Chandra \cite[Lemma 42]{HC70} mentions that (in the characteristic $0$ case) the proof is the same as in the real case. The proof in the real case depends on the substitution rule from the theory of analytic manifolds. For the readers, that are not quite familiar with analytic/differential geometry, we give a couple of references and results, before we start proving the Weyl integration formula.

We refer to \cite{Se65} for the definitions of an analytic manifold and of a tangent space in a point.\\
Since $G(\F)$ is Zariski-dense in $G$ and $G$ is an affine linear algebraic group, we have an analytic structure on $G(\F)$, by \cite[\S 3.1]{PR94}.\\
For the integration on the group $G$ we refer to \cite[\S 3.5]{PR94}. The $G$-invariant measure on $G$ can be defined as follows: Let $\omega \in \bigwedge^nT^*_e(G)$ with $\omega\not=0$, then define $\omega_g:= L^*_{g^{-1}}\omega$. The map $g\mapsto \omega_g$ is a $G$-invariant $n$-form on $G$ and leads to a $G$-invariant measure.\\

Let $H$ be an analytic submanifold of $G$. There is an unique analytic structure on the quotient $G/H$ making $\pi:G\rightarrow G/H$ into a submersion by \cite[LG \S 4.5]{Se65}. Now $G$ is a so-called right principal $H$-bundle over the base $G/H$, \cite[LG \S 4.5, Theorem 6]{Se65}. In particular there is for every $b\in G/H$ an open set $U_b$ and an analytic isomorphism $\tau : \pi^{-1}U_b \rightarrow U_b\times H$, such that $\tau(x)= (\pi(x),\phi(x))$. The function $\tau$ is called the local trivialisation.\\

As shown in \cite[\S 3.13]{DK00} we can choose the differential forms on $G$, $G/T$ and $T$ in such a way that
\[ \int_Gf(g)dg = \int_{T} \int_{G/T} f(gt) dgdt.\]
If $\tau : \pi^{-1}U\rightarrow T\times U$ is a local trivalization, then
\[ \int_{\pi^{-1}U}f(g)dg = \int_{U} \int_{T} f(\tau^{-1}(u,t)) dgdt.\]
\begin{proof}[Proof of Theorem \ref{WIF}] The proof of \cite[Theorem 3.14.1]{DK00} in the real compact case works in this case as well. We only take a different definition of the subspace $\mathfrak{q}$. It has to be a $Ad(T)$-stable $\F$-linear subspace of $\gl$, which is complementary to $\tl$. When the characteristic of $\F$ is zero, the resulting subspace is the same.

The Lie algebra $\tl$ of $T$ has a complementary $Ad(T)$-invariant space $\mathfrak{q}$ defined over $\F$. We define $\mathfrak{q} := \bigoplus_{\alpha\in R(G,T)} \gl_\alpha$. Then $\mathfrak{q}$ is $Ad(T)$-invariant. If $T$ is $\F$-split, then clearly $\gl_\alpha$ is defined over $\F$ and hence also $\mathfrak{q}$ is defined over $\F$. Take $\E$ a Galois-extension of $\F$ such that $T$ is $\E$-split. Let $\Gamma := \text{Gal}(\E:\F)$. Since $T$ is $\E$-split, $\mathfrak{q}$ and $\gl_\alpha$ are defined over $\E$. Thus $\mathfrak{q}$ is defined over $\F$ if and only if it is $\Gamma$-invariant.
Let $x\in \gl_\alpha$, then for all $\gamma\in \Gamma$:
\[ t\gamma(x)t^{-1} = \gamma( \gamma^{-1}(t)x\gamma^{-1}(t^{-1})) = \gamma(\alpha(\gamma^{-1}(t))x)=\gamma(\alpha(\gamma^{-1}(t)))\gamma(x). \]
Thus $\gamma(\alpha(\gamma^{-1}(\cdot)))\in R(G,T)$, hence $\gamma(x)\in \mathfrak{q}$.
\end{proof}

\bibliographystyle{alpha}

\end{document}